\documentclass[article]{IEEEtran}
\usepackage{cite}
\usepackage{color}
\usepackage[pdftex]{graphicx}
\graphicspath{{fig/}{jpeg/}}
\usepackage[cmex10]{amsmath}
\usepackage{amssymb}
\usepackage{algorithm}
\usepackage{algpseudocode}
\usepackage{amsmath,amssymb,lipsum}
\usepackage{hyperref}
\usepackage{comment}
\usepackage{graphicx}
\usepackage[font=small]{caption}
\usepackage{subcaption}
\usepackage{array}
\input{mysymbol.sty}
\usepackage{needspace}

\usepackage{amsthm}
\usepackage{tikz}
\usetikzlibrary{shapes,arrows}
\usepackage{bbm}

\newtheorem{theorem}{Theorem}
\newtheorem{definition}{Definition}
\newtheorem{lemma}{Lemma}
\newtheorem{corollary}{Corollary}

\theoremstyle{definition}
\newtheorem{assumption}{Assumption}
\newtheorem{remark}{Remark}

\title{Stochastic Artificial Potentials for Online Safe Navigation}
\author{Santiago Paternain and Alejandro Ribeiro
\thanks{Work in this paper is supported by NSF CNS-1302222 and ONR N00014-12-1-0997. The authors are with the Department of Electrical and Systems Engineering, University of Pennsylvania, 200 South 33rd Street, Philadelphia, PA 19104. Email: \{spater, aribeiro\}@seas.upenn.edu.}}

\begin{document}

\maketitle
\thispagestyle{empty}
\pagestyle{empty}

%
\begin{abstract}
  Consider a convex set of which we remove an arbitrarily number of disjoints convex sets -- the obstacles -- and a convex function whose minimum is the agent's goal. We consider a local and stochastic approximation of the gradient of a Rimon-Koditschek navigation function where the attractive potential is the convex function that the agent is minimizing. In particular we show that if the estimate available to the agent is unbiased convergence to the desired destination while obstacle avoidance is guaranteed with probability one under the same geometrical conditions than in the deterministic case. Qualitatively these conditions are that the ratio of the maximum over the minimum eigenvalue of the Hessian of the objective function is not too large and that the obstacles are not too flat or too close to the desired destination. Moreover, we show that for biased estimates a similar result holds under some assumptions on the bias. These assumptions are motivated by the study of the estimate of the gradient of a Rimon-Koditschek navigation function for sensor models that fit circles or ellipses around the obstacles. Numerical examples explore the practical value of these theoretical results.
\end{abstract}
%
%

\section{Introduction}
The problem of navigating towards a desired goal configuration has been extensively studied in the robotics community. In the particular case where the set of available configurations to the robot is convex it is possible to reach the desired configuration by implementing a gradient controller (see e.g. \cite{hirsch2004differential}). The main advantages of such controllers are their simplicity and the fact that they rely only on local information, this is, in the gradient of a function whose minimum is the goal configuration. 

A much more complex setting is one in which the workspace is cluttered by obstacles that must be avoided by the agent. Solutions to this problem have been provided in the form of artificial potentials, see for instance \cite{koditschek1990robot,  rimon1992exact,khatib1980commande, Khatib:1986:ROA:6806.6812, lozano1987handey, newman1987high,barraquand1992numerical, khosla1988superquadric, barraquand1990monte,connolly1990path,krogh1984generalized,warren1989global,lionis2007locally,lionis2008towards,filippidis2011adjustable,filippidis2012navigation,filippidis2013navigation}. The main idea of this approach is to combine the attractive potential with repulsive fields that push the agent away of the boundary of the obstacles. With proper design -- and restring the geometry of the obstacles to certain classes -- it is possible to construct a potential that attains its maximum at the boundary of the obstacles and with a unique minimum at the goal configuration. Therefore ensuring non collision with the obstacles and convergence to the desired destination from almost every initial configuration when following the negative gradient of this potential. The existence guarantees of such functions -- termed navigation functions -- is highly dependent on the geometry of the free space. For instance for artificial potentials of the  Rimon-Koditschek form introduced in \cite{koditschek1990robot} the above properties can be guaranteed in the case of focally admissible obstacles \cite{filippidis2013navigation} of which spherical worlds considered in the original work \cite{koditschek1990robot} are a particular case. This said, by implementing a suitable diffeomorphism it is possible to extend the results of \cite{koditschek1990robot} to star worlds \cite{rimon1992exact,rimon1991construction} thus extending considerably the families of free spaces that can be navigated. Different families of navigation functions can be constructed, such is the case of navigation function based in harmonic functions which allow navigation in topologically complex three dimensional spaces \cite{loizou2011closed,loizou2012navigation}. The latter construction needs the free space to be diffeomorphically mapped to a reference world. In that sense the navigation framework lack the advantage of pure gradient controllers: these cannot be implemented locally as they necessitate access to some amount of global information. Efforts in overcoming this limitation have been pursued, in particular through the use of polynomial navigation functions in the case of two-dimensional configuration spaces with convex obstacles \cite{lionis2007locally, lionis2008towards} and in $n$ dimensional configuration spaces with spherical obstacles \cite{filippidis2011adjustable}.
%

In the navigation function framework typically the goal configuration is provided to the robot and therefore a rotational symmetric attractive potential can be considered. However, in some settings it is desirable to provide the configuration goal as the minimum -- or maximum -- of an objective function instead of the configuration itself. Consider for instance the hill climbing problem in which an agent can sense its way ``up'' by following the slope of the terrain estimated by an inertial measurement unit (IMU). It is more reasonable to solve the problem as navigating towards the top of the hill following its slope --and reaching a point where the slope becomes zero-- as compared as navigating towards a given location. This is especially true if the height profile of the hill is unknown or if the interest is on building a system that is independent of the particular hill under consideration. Generally speaking reaching the minimum of an unknown function is a desirable capability for robots to perform complex missions such as environmental monitoring \cite{ogren2004cooperative,sukhatme2007design}, surveillance and reconnaissance \cite{rybski2000team} and search and rescue operations \cite{kumar2004robot}. The problem of navigating towards the minimum of a convex function in a space with convex holes is studied \cite{PaternainEtal15}, where generic conditions are presented to ensure that a Rimon-Koditschek navigation function can be constructed when the attractive potential is a generic convex function rather than the squared of the distance to a desired configuration for a workspace with convex obstacles. The qualitative implication of this conditions is that Rimon-Koditschek have a unique minimum when one of the following conditions are met. (i) The condition number of the Hessian of the attractive potential is not large and the obstacles are not too flat. (ii) The distance from the obstacles' boundary to the minimum of the attractive potential is large relative to the size of the obstacle. These conditions are compatible with the definition of sufficiently curved worlds in \cite{filippidis2012navigation}.    
 
In \cite{PaternainEtal15} it is assumed that the information about the objective function and the obstacles is exact. However, this is not the case in systems where the magnitudes that the robot needs to build the navigation function are gathered by sensors and therefore the measurements have errors in the form of noise. In that sense the objective of this work is to generalize the results in \cite{PaternainEtal15} to stochastic scenarios, understood as a setting in which the sensorial information available to the agent comes from a probability distribution instead of being deterministic (Section \ref{sec_problem_formulation}). In particular we show that if the agent is able to construct an unbiased estimate of the gradient of the navigation function, convergence to the minimum of the objective function can be ensured with probability one as well as collision avoidance (Theorem \ref{theo_unbiased} Section \ref{sec_unbiased}). Moreover, there might be a mismatch between the model that the agent has of the environment and the real one. This mismatch translates into the fact that estimates of the gradient of the navigation function are not unbiased. Hence we devote Section \ref{sec_biased} to this end. In particular, we show that if in a neighborhood of the saddle points of the navigation function the bias is small the same theoretical guaranties as in the unbiased case can be provided (Theorem \ref{theo_unbiased}). The previous technical hypothesis is motivated by the study of particular sensor models in Section \ref{sec_sensor_model}. The practical implications of these theoretical conclusions are explored in numerical simulations (Section \ref{sec_numerical_example}) in which we consider the problem of reaching the minimum of non rotational symmetric potentials in a space where the obstacles are ellipses (Section \ref{sec_elliptical_obstacles}) and where the obstacles are egg shaped as an example of a generic convex obstacle (Section \ref{sec_egg_shaped}).

%
\section{Problem formulation}\label{sec_problem_formulation}
In this work we are interested in navigating towards the minimum of a convex potential in a space with convex holes in cases where the information available to the agent about the potential and the space is local and inexact. To be formal, define the workspace $\ccalX \subset \mathbb{R}^n$ as a non empty convex compact set and consider a set of $m\in \mathbb{N}$ obstacles $\ccalO_i \subset \ccalX$  that we define as non empty, open, strongly convex sets with smooth boundary $\partial \ccalO_i$. The obstacles are such that they do not intersect with each other or with the boundary of the workspace. We formalize these assumptions next. 
%
%
\begin{assumption}[\bf Obstacles do not intersect]\label{as_obstacles}
The workspace and the obstacles are such that the obstacles and its boundaries are contained in the interior of the workspace
\begin{equation}
\left(\ccalO_i\cup \partial\ccalO_i \right)\subset int(\ccalX) \quad \mbox{for all} \quad i=1\ldots m,
\end{equation}
and the obstacles do not intersect with each other
\begin{equation}
\left(\ccalO_i\cup \partial\ccalO_i \right)\cap \left(\ccalO_j\cup \partial\ccalO_j \right)= \emptyset\quad  \forall i,j=1\ldots m, i\neq j.
\end{equation}
\end{assumption}
%
The free space $\ccalF$ represents the points of the workspace that are accessible to the agent, i.e., the set difference between the workspace and the obstacles. We formally define this set next.  
%
%
\begin{definition}\label{def_freespace}
The free space $\ccalF\subset \mathbb{R}^n$ is the set given by
\begin{equation}\label{eqn_freespace}
\ccalF = \ccalX \setminus \bigcup_{i=1}^m \ccalO_i.
\end{equation}
\end{definition}
%
%
Let $f_0: \ccalX\to \mathbb{R}_+$ be a convex function such that its minimum is the agent's goal. Then the problem of interest is  to navigate the free space $\ccalF$ towards the minimum of the convex potential $f_0(x)$ from all initial positions. Formally, this is finding a sequence 
\begin{equation}\label{eqn_navigation_problem}
\left\{x_t \in \ccalF, t \in \mathbb{N}\cup \left\{0\right\}\right\} \quad \mbox{such that}  \quad \lim_{t\to\infty} x_t =x^*,
\end{equation}
where $x^* = \argmin f_0(x)$. For such a problem to be feasible we need the minimum of the potential to be in the free space. We also require the objective function to be twice continuously differentiable and strongly convex. We formalize these assumptions about the objective function next.
\begin{assumption}[\bf Objective function]\label{assum_objective_function}
The objective function $f_0(x)$ is such that:
\begin{mylist}
\item[\bf Optimal point]
The minimum $x^*$ of the objective function is such that $f_0(x^*)\geq 0$ and it is in the interior of the free space,
\begin{equation}
x^* \in int(\ccalF).
\end{equation} 
\item[\bf Twice continuously differentiable and strongly convex] 
The objective function is twice continuously differentiable and strongly convex in $\ccalX$. These assumptions in addition to the fact that the workspace is compact imply that the eigenvalues of the Hessian $\nabla^2 f_0(x)$ are contained in the interval $[\lambda_{\min}, \lambda_{\max}]$ for all $x\in\ccalF$, with $0 < \lambda_{\min}$.
\end{mylist}
\end{assumption}
In cases where exact information about the objective function and complete information about the obstacles is available to the agent, it is possible --under mild conditions about the geometry of the free space and the objective function-- to build a navigation function \cite{PaternainEtal15}. An agent that follows the flow given by the negative gradient of a navigation function converges to the destination $x^*$ without running into the free space boundary for a set of initial conditions that is dense in the free space\cite{koditschek1988strict}. Thus solving problem \eqref{eqn_navigation_problem}. For completeness we provide here the definition of a navigation function as well as a different characterization of the free space that is useful to the navigation function framework.
%
\begin{definition}[\bf{Navigation Function}]\label{def_navigation_function} 
Let $\mathcal{F} \subset \mathbb{R}^n$ be a compact connected analytic manifold with boundary. A map $\varphi : \mathcal{F} \rightarrow [0,1]$, is a navigation function in $\ccalF$ if:
\begin{mylist}
\item[{\bf Differentiable.}] It is twice continuously differentiable in $\mathcal{F}$.
\item[{\bf Polar at $x^*$.}] It has a unique minimum at $x^*$ which belongs to the interior of the free space, i.e., $x^* \in \mbox{int}(\mathcal{F})$.
\item [{\bf Morse.}] It has non degenerate critical points on $\ccalF$.
\item [{\bf Admissible.}] All boundary components have the same maximal value, namely $\partial \mathcal{F} = \varphi^{-1}(1)$\footnote{For a function $f(x)$ we denote its inverse by $f^{-1}(x)$.}.
\end{mylist}
\end{definition}
%
%
Since the workspace $\ccalX$ is a convex set, there exists a concave function $\beta_0: \mathbb{R}^n \to \mathbb{R}$ such that $x\in \ccalX$ if and only if $\beta_0(x)\geq 0$. Such a function exists because super level sets of concave functions are convex. Likewise we can define convex functions $\beta_i(x): \mathbb{R}^n\to \mathbb{R}$ for $i=1\ldots m$ such that $\beta_i(x) \leq 0$ if and only if $x\in \ccalO_i \cup \partial \ccalO_i$. Since the obstacles $\ccalO_i$ are smooth and strongly convex the Hessian of the function $\beta_i(x)$ is well defined and its eigenvalues are lower bounded by $\mu^i_{\min}>0$. Define then the following product function $\beta:\mathbb{R}^n \to \mathbb{R}$  
\begin{equation}\label{eqn_beta}
\beta(x) = \prod_{i=0}^m \beta_i(x).
\end{equation}
The interest in defining the above function is that it is possible to characterize the free space as the set for which $\beta(x)$ is nonnegative, in particular its boundary are the points satisfying $\beta(x) =0$. With this characterization of the free space one can define the following Rimon Koditschek artificial potential 
\begin{equation}\label{eqn_navigation_function}
\varphi_k(x) = \frac{f_0(x)}{\left(f_0^k(x) + \beta(x)\right)^{1/k}},
\end{equation}
where $k>0$ is an order parameter. It can be shown that for large enough $k$ under mild assumptions on the condition number of the Hessian of the objective functions and the geometry of the free space the above artificial potential is a navigation function. These conditions are given in the following Theorem \cite{PaternainEtal15}.
\begin{theorem}\label{theo_navigation_function}
Let $\mathcal{F}$ be the free space defined in \eqref{eqn_freespace} verifying Assumption \ref{as_obstacles}, and let $\varphi_k: \mathcal{F} \rightarrow [0,1]$ be the function defined in \eqref{eqn_navigation_function}. Let $\lambda_{\max}$, $\lambda_{\min}$ be the bounds from Assumption \ref{assum_objective_function} and $\mu^i_{\min}$ the minimum eigenvalue of the Hessian of $\beta_i(x)$. Furthermore let the following inequality hold for all $i=1..m$
\begin{equation}\label{eqn_condition_general}
\frac{\lambda_{\max}}{\lambda_{\min}} \frac{\nabla \beta_i(x_s)^T(x_s-x^*)}{\|x_s - x^*\|^2}< \mu^i_{\min},
\end{equation}
where $x_s\in \partial \ccalO_i$ . Then there exists a constant $K$ such that if $k>K$, $\varphi_k(x)$ is a navigation function with minimum at $x^*$ if $f_0(x^*) =0$ and with minimum arbitrarily close to $x^*$ if $f_0(x^*) \neq 0$. 
\end{theorem}
\begin{proof}
See Theorem 2 in \cite{PaternainEtal15}.
\end{proof}
Theorem \ref{theo_navigation_function} provides a condition on the obstacles and the objective function for which $\varphi_k(x)$ is a navigation function for sufficiently large $k$. The condition has to be satisfied for all the points lying in the boundary of an obstacle. Notice however that the product $\nabla \beta_i(x_s)^T(x_s - x^*)$ is negative if $\nabla \beta_i(x_s)$ and $x_s-x^*$ point in opposite directions, meaning that the condition can be violated only by points in the boundary of the obstacle that are behind the obstacle as seen from the minimum point. In that case the worst scenario is when $\nabla \beta_i(x_s)$ is aligned with $x_s - x^*$. In this case it is of interest that the gradient $\nabla \beta_i(x_s)$ is not too large with respect to the minimum eigenvalue $\mu^i_{\min}$, i.e., the obstacle is not too flat. On the other hand we want the ratio $1/\|x_s-x^*\|$ to be small in order to satisfy \eqref{eqn_condition_general}. This ratio being small means that the destination $x^*$ is not too close to the boundary of the obstacle. Finally, condition \eqref{eqn_condition_general} is easier to satisfy when the ratio $\lambda_{\max}/\lambda_{\min}$ is close to one, meaning that the closer the level sets of the objective function are to spheres, the easier is to navigate the environment. In summary, the simplest navigation problems have obstacles and objective function whose level sets are close tho spheres and minima that are not close to the boundary of the obstacles.

While the navigation function approach provides a provable way of navigating towards the minimum of a convex potential in a cluttered workspace, its drawback is that it needs a complete characterization of the obstacles to build the function $\varphi_k(x)$ defined in \eqref{eqn_navigation_function}. Moreover, to ensure that the agent is moving in the direction of the negative gradient of the navigation function, the measurements of the objective function and the obstacles need to be exact. In this work we relax these assumptions by considering only local and stochastic information. Formally, let $\left( \Omega,\ccalG, P\right)$ be a probability space and define the following filtration defined as a sequence of increasing sigma algebras $\left\{\emptyset, \Omega\right\} = \ccalG_0 \subset \ccalG_1 \subset \ldots \subset \ccalG_t \subset \ldots \subset \ccalG$. For each $t \geq 0$, define a random vector $\theta_t $ to be $\ccalG_t$ measurable. Then at each time $t\in\mathbb{N}$ for a given position in the free space $x_t \in \ccalF$ the agent is able to compute a biased estimate of the gradient of the navigation function $\hat{g}_t(x_t,\theta_t)$ satisfying
\begin{equation}\label{eqn_estimate_form}
\E{\hat{g}_t(x_t,\theta_t)\Big| \ccalG_t} = \alpha(x)\left(\nabla \varphi_k(x)+b_k(x)\right),
\end{equation}
where $\alpha:\ccalF \to \mathbb{R}$ is a strictly positive differentiable function and $b_k:\ccalF \to \mathbb{R}^n$ is piece-wise differentiable. As it will be explored in Section \ref{sec_sensor_model} the bias $b_k(x)$ accounts for a mismatch between the real free space and the one that the robot is able to estimate given some belief about the environment. This mismatch is the consequence of using local information about the free space. Drawing inspiration from the deterministic scenario we propose a stochastic gradient descent scheme to solve \eqref{eqn_navigation_problem} using only local and stochastic information in which the agent updates its configuration recursively as
\begin{equation}\label{eqn_gradient_descent}
x_{t+1} = x_t - \varepsilon_t \hat{g}_t(x_t,\theta_t),
\end{equation}
where $\varepsilon_t$ is a step size assumed to be not summable and square summable. Typically one can select the step size as $\varepsilon_t = \varepsilon_0/ (1+\zeta t)$, where $\varepsilon_0$ is the initial step size and $\zeta$ controls the rate at which the step size is decreased. We formalize the assumption on he step size for future reference. 
%
\begin{assumption}\label{assumption_step_size}
The step size $\varepsilon_t$ for the update \eqref{eqn_gradient_descent} is a positive and strictly decreasing sequence that satisfies 
 \begin{equation}
 \sum_{t=0}^\infty \varepsilon_t =\infty, \quad \sum_{t=0}^\infty \varepsilon_t^2 < \infty.
 \end{equation}
\end{assumption}
%
The main contribution of this work is to show that an agent operating in a workspace with convex holes, that is given an estimate of the form \eqref{eqn_estimate_form} is able to reach the minimum of a unknown convex function without running into the free space boundary with probability one (Section \ref{sec_biased}). Before presenting this result, in Section \ref{sec_sensor_model} we consider a sensor model from which an estimate satisfying \eqref{eqn_estimate_form} arises and we present a preliminary result for unbiased estimates (Section \ref{sec_unbiased}).
%
\section{Sensor Model Examples}\label{sec_sensor_model}
In this section we propose an estimate of the gradient of a Rimon-Koditschek navigation function based on local and stochastic observations about the objective functions and the obstacles. The estimate proposed is based in the fact that the direction of the gradient of the potential defined in \eqref{eqn_navigation_function} is given by the following expression
\begin{equation}\label{eqn_navigation_direction}
\beta(x)\nabla f_0(x) - \frac{f_0(x) \nabla \beta(x)}{k}.
\end{equation}
The above fact can be conclude after differentiating the expression \eqref{eqn_navigation_function} and noticing that the terms that multiply \eqref{eqn_navigation_direction} are strictly positive. Since the objective function is typically a physical magnitude that must be minimized or maximized one can assume that the robot has estimates of the function $f_0(x)$ and its gradient at the current location. For instance in the problem of climbing a forested hill the function $f_0(x)$ represents the height profile of the hill. Using a GPS the agent is able to have a measure of the height at the current location and with an inertial measurement unit (IMU) it is possible to estimate the slope of the hill understood as the gradient of the height profile function $f_0(x)$. Denote these estimates at time $t$ by $\hat{f}_0(x_t,\theta_t)$ and $\hat{\nabla}f_0(x_t,\theta_t)$, where $\theta_t$ is a random vector measurable with respect to the sigma algebra $\ccalG_t$. In order to estimate the obstacles -- the trees in the hill climbing problem-- the agent may have information available gathered by a range finder. In this case depending on the belief that the agent has about the world there exists different forms of estimating the obstacles of which we discuss two examples next. Before doing that we define the set of obstacles that can be measured at a given position $x$. Due to physical limitations like the range of the sensor or the fact that obstacles can be ``hidden'' behind others the agent is not able to sense all the obstacles at a given position $x$. In that sense we define the set obstacles that can be estimated as those obstacles that are at a distance smaller than a given limit $c$
\begin{equation}\label{eqn_awareness_set}
\ccalA_c(x) = \left\{ i=1\ldots m \Big| d_i(x) \leq c\right\},
\end{equation}
where $d_i(x)$ is the distance to the $i$--th obstacle. 
\subsection{Circle Fitting}\label{sec_circles}
We consider the case where the belief that the robot has about the free space is that obstacles are spherical. Online estimation of distance, direction and curvature of the obstacles has been studied in the literature \cite{paper:de_wall_following_2013}. Denoting these quantities corresponding to the $i$--th obstacle by $d_i(x)$, $\bbn_i(x)$ and $R_i(x)$, the agent assumes the obstacle function to be
\begin{equation}
\tilde{\beta}_i(x)=d_i^2(x)+ 2R_i(x)d_i(x),
\end{equation}
and the assumed gradient of the function is of the form
\begin{equation}
\tilde{\nabla\beta}_i(x) = 2\left(d_i(x)+R_i(x)\right)\bbn_i(x).
\end{equation}
In particular observe that if the free space is indeed a spherical world the functions $\tilde{\beta}_i(x)$ and $\beta_i(x)$ are identical as well as $\tilde{\nabla\beta}_i(x)$ and $\nabla \beta_i(x)$. Denoting the estimates of the distance, direction and curvature of the $i$-th obstacle respectively by $\hat{d}_i(x_t,\theta_t)$, $\hat{\bbn}_i(x_t,\theta_t)$ and $\hat{R}_i(x_t,\theta_t)$, one can define an estimate of the function corresponding to the obstacle $\ccalO_i$ as 
\begin{equation}
\hat{\beta}(x_t,\theta_t) = \hat{d}_i^2(x_t,\theta_t) +2\hat{R}_i(x_t,\theta_t)\hat{d}_i(x_t,\theta_t),
\end{equation} 
and its gradient by
\begin{equation}
\hat{\nabla \beta}_i(x_t,\theta_t) =2 \left(\hat{d}_i(x_t,\theta_t) +\hat{R}_i(x_t,\theta_t)\right)\hat{\bbn}_i(x_t,\theta_t).
\end{equation}
With this information available a natural possibility inspired in \eqref{eqn_navigation_direction} is to define the estimate of the direction of the gradient of the navigation function as 
\begin{equation}\label{eqn_estimate_navigation_direction}
\begin{split}
\hat{g}_t(x_t,\theta_t) &:= \hat{\nabla}f_0(x_t,\theta_t)\prod_{i\ccalA_c(x_t)}\hat{\beta}_i(x_t,\theta_t)\\
&- \frac{\hat{f}_0(x_t,\theta_t)}{k}\sum_{i\in\ccalA_c(x_t)} \hat{\nabla\beta}_i(x_t,\theta_t)\prod_{j\in\ccalA_c(x_t), j\neq i} \hat{\beta}_j(x_t,\theta).
\end{split}
\end{equation}
By taking the expectation of the estimate with respect to the sigma algebra $\ccalG_t$ and assuming independence across estimates it is possible to show that the estimate \eqref{eqn_estimate_navigation_direction} satisfies \eqref{eqn_estimate_form}. Observe that if the estimates corresponding to the objective function and the obstacles are bounded -- which is the case in practical applications-- the estimate of the direction of the gradient has bounded norm. Further notice, that when an agent is close to the obstacle $\ccalO_i$  we have that $\beta_i(x_t) \approx 0$. Therefore, the direction $\hat{g}_t(x_t,\theta_t)$ is approximately given by
\begin{equation}
\hat{g}_t(x_t,\theta_t) \approx -\frac{\hat{f}_0(x_t,\theta_t) }{k} \prod_{j\in\ccalA_c(x_t), j\neq i} \hat{\beta}_j(x_t,\theta) \hat{\nabla \beta}_i(x_t,\theta_t).
\end{equation}
The above means that the update direction proposed in \eqref{eqn_gradient_descent} points outwards the $i$-th obstacle when the agent is close to it. These observations made for this particular estimator are presented as Assumption \ref{assumption_estimator_model} in Section \ref{sec_general_assumptions} for the general case. We next devote our attention to the properties of the bias $b_k(x)$. Let $\sigma_{d_i}^2(x)$ be the variance of the estimate of the distance to obstacle $\ccalO_i$. For the estimate defined in \eqref{eqn_estimate_navigation_direction} the bias $b_k(x)$ takes the particular form of 
\begin{equation}
\begin{split}
b_k(x) &= \frac{f_0(x)\beta(x)}{k\left(f_0(x)^k+\beta(x)\right)^{1+1/k}}\times\\
&\left(\sum_{i=0}^m\frac{\nabla \beta_i(x)}{\beta_i(x)}-\sum_{i\in\ccalA_c(x)}\frac{\tilde{\nabla \beta}_i(x)}{\tilde{\beta}_i(x)+\sigma_{d_i}^2(x)}\right).
\end{split}
\end{equation}
Observe that the bias depends upon three main factors, the limitation in the number of obstacles that can be measured, the difference between the free space and the belief of the agent and the variance of the estimation of the distance to the obstacles. In the particular case where the wolrd is spherical, the agent is able to sense all the obstacles and the distance to the obstacle is know exactly -- or an unbiased estimate of the distance squared is available-- the estimator is unbiased. In the general case it is possible to show that as long as the variance $\sigma_{d_i}^2(x)$ vanishes fast enough when $x$ approaches the boundary of $\ccalO_i$ we have that 
\begin{equation}\label{eqn_bounded_difference}
\left\|\sum_{i=0}^m\frac{\nabla \beta_i(x)}{\beta_i(x)}-\sum_{i\in\ccalA_c(x)}\frac{\tilde{\nabla \beta}_i(x)}{\tilde{\beta}_i(x)+\sigma_{d_i}^2(x)}\right\| \leq B',
\end{equation}
for all $x\in\ccalF$, where $B'$ is a nonnegative constant. The fact that the variance of the estimate of the distance vanishes translates in the fact that the closest the agent is to an obstacle the better it can be estimated. In particular, the estimation in the boundary is exact. Since the gradient of $\varphi_k(x)$ has a factor of $1/\left(f_0(x)^k+\beta(x)\right)^{1+1/k}$ it is more convenient to work with the following scaling of the bias 
\begin{equation}
\tilde{b}_k(x) = \left(f_0(x)^k+\beta(x)\right)^{1+1/k}b_k(x).
\end{equation}
Some consequences of the bias vanishing in the boundary of the free space are that for any $x\in \partial \ccalF$ we have $\tilde{b}_k(x) =b_k(x)= 0$ since $\beta(x) =0$. Further observe that the norm of $\tilde{b}_k(x)$ is decreasing at the rate $1/k$ for any point in the interior of the free space and in particular $\lim_{k\to \infty} \tilde{b}_k(x)=0$. Moreover, under this model the function $\tilde{b}_k(x)$ is piece-wise twice differentiable and the discontinuities are due to changes in the set $\ccalA_c(x)$, this is either when a new obstacle is sensed or when an obstacle cannot be sensed anymore. Therefore, the discontinuities occur away from the obstacles.
Further observe that since $\tilde{b}_k(x)$ is decreasing with $k$ and because $\lim_{k\to\infty}\|\tilde{b}_k(x)\|=0$ we have that the region where $\nabla \varphi_k(x)^T\left(\nabla \varphi_k(x)+b_k(x)\right)\leq 0$ are disjoint regions around the critical points of $\varphi_k(x)$ for large enough $k$. Let $x_c$ be a saddle point of $\varphi_k(x)$ and define the direction $v = \nabla \beta(x_c) / \|\nabla \beta(x_c)\|$ and $v_\perp$ a unit vector satisfying $v^T v_\perp=0$. One can show that if the obstacles are spherical the quotient of the quadratic form of the Jacobian of $b_k(x)$ at $x_c$ over the quadratic form of the Hessian of $\varphi_k(x)$ at $x_c$ is such that
\begin{equation}\label{eqn_quotient_jacobians1}
\left\|\frac{v^T Jb_k(x_c)v}{v^T\nabla^2\varphi_k(x_c)v}\right\| = O(1/k),
\end{equation}
and
\begin{equation}\label{eqn_quotient_jacobians2}
\left\|\frac{v_\perp^T Jb_k(x_c)v_{\perp}}{(v_\perp)^T\nabla^2\varphi_k(x_c)v_\perp}\right\| = O(1/k),
\end{equation}
where $O(1/k)$ is a function whose limit $\lim_{k\to \infty} O(1/k)k$ is a positive constant. It is also worth noticing that the saddle points $x_c$ of $\nabla \varphi_k(x)$ satisfy that $\beta(x_c)\leq L/k$ where $L$ is a non-negative constant (see Lemma 3 of \cite{PaternainEtal15}) the scaled bias satisfies $\tilde{b}_k(x_c) =O(1/k^2)$. The interpretation of the previous fact is that at the critical points of $\varphi_k(x)$, the $C^1$ norm\footnote{Given a vector field $f(x)$ we denote its $n$-derivative by $D^{(n)}f(x)$. We define the $C^n$ norm of a vector field $f(x)$ in a manifold $M$ as $\|f(x)\|_{C^n}=\sup_{x\in M}\left\{\|f(x)\|,\|Df(x)\|,\ldots, \|D^{(n)}f(x)\| \right\}$.} of the bias is small compared to that of the vector field $\nabla \varphi_k(x)$. In particular, for large enough $k$ in a neighborhood around a saddle point of $\varphi_k(x)$ the eigenvalues of the Jacobian of $\nabla \varphi_k(x) +b_k(x)$ have the same sign as those of the Hessian of $\varphi_k(x)$, therefore having the same stability properties. These observations about the bias for the particular estimate here presented are summarized under Assumption \ref{assumption_bias} for the generic case (c.f. Section \ref{sec_general_assumptions}). 
\subsection{Ellipse Fitting}\label{sec_ellipse}
A different approach for obstacle estimation is to fit ellipses around the obstacles instead of circles. In this case the functions defining the obstacles take the form 
\begin{equation}
\tilde{\beta}_i(x) = \left(x-x_i\right)^TA_i(x-x_i) - r_i^2,
\end{equation}
where $A_i$ is a symmetric $n\times n$ matrix. Thus, in order to fit ellipses around the obstacles one needs to estimate $(n-1)^2/2+n$ parameters corresponding to the matrix $A_i$, $n$ parameters corresponding to the center of the ellipses $x_i$ and one parameter corresponding to the scaling $r_i$. This is a drawback compared to the case of the circle where only its radius was needed, yet it reduces the mismatch between the model and the true environment for a larger class of obstacles. Under this model and assuming that unbiased estimates of the discussed quantities are available one can estimate the obstacle function as
\begin{equation}
\begin{split}
&\hat{\beta}_i(x_t,\theta_t) = -\hat{r}_i^2(x_t,\theta_t)+ \\
& \left(\hat{x_t}(x_t,\theta_t)-\hat{x}_i(x_t,\theta_t)\right)^T\hat{A}_i(x_t,\theta_t)\left(\hat{x_t}(x_t,\theta_t)-\hat{x}_i(x_t,\theta_t)\right),
\end{split}
\end{equation}
and its gradient as 
\begin{equation}
\hat{\nabla \beta}_i(x_t,\theta_t) = 2\hat{A}_i(x_t,\theta_t)\left(\hat{x_t}(x_t,\theta_t)-\hat{x}_i(x_t,\theta_t)\right).
\end{equation}
As discussed in the previous section \eqref{eqn_quotient_jacobians1} and \eqref{eqn_quotient_jacobians2} hold when the obstacles are spherical, likewise when considering ellipses as hallucinated obstacles \eqref{eqn_quotient_jacobians1} and \eqref{eqn_quotient_jacobians2} holds for obstacles that do not differ much from ellipsoids. 

\subsection{General Model Assumptions}\label{sec_general_assumptions}
We summarize the observations about the estimate of the gradient of the navigation function $\hat{g}_t(x_t,\theta_t)$ for the particular models described in Sections \ref{sec_circles} and \ref{sec_ellipse} under the following assumptions for a generic estimate satisfying \eqref{eqn_estimate_form}. 
%
\begin{assumption}\label{assumption_estimator_model}
The estimate of the gradient of the navigation function $\hat{g}(x_t,\theta_t)$ is 
\begin{mylist}
\item[\bf Bounded] There exists a strictly positive constant $B$ such that for all $x\in\ccalF$ and for all $\theta$ we have that 
\begin{equation}
\| \hat{g}(x,\theta)\| \leq B.
\end{equation}
\item[\bf Points outwards the obstacles]
For each obstacle $\ccalO_i$ there exists a constant $\gamma_i>0$ such that if $d_i(x) < \gamma_i$ we have for all $\theta$   
\begin{equation}
-\hat{g}(x,\theta)^T \nabla\beta_i(x)>0,
\end{equation}   
where $d_i(x)$ denotes the distance to the obstacle $\ccalO_i$.
\item[\bf Biased]
Let $\alpha(x):\ccalF\to \mathbb{R}_{++}$ be a differentiable function bounded away from zero and let $b_k(x):\ccalF\to\mathbb{R}^n$ be piece-wise differentiable on the free space and let $\varphi_k(x)$ be the function defined in \eqref{eqn_navigation_function}. Then the expected value of the estimate $\hat{g}_t(x_t,\theta_t)$ with respect to the sigma algebra $\ccalG_t$ satisfies 
\begin{equation}
\E{\hat{g}_t(x_t,\theta_t)\Big| \ccalG_t} = \alpha(x_t)\left(\nabla \varphi_k(x_t) + b_k(x_t)\right). 
\end{equation}
\end{mylist}
\end{assumption}
%
\begin{assumption}\label{assumption_bias}
The bias $b_k(x)$ defined in \eqref{eqn_estimate_form} is piece-wise differentiable on the free space and has the following properties. 
\begin{mylist}
\item[\bf Unbiased at the boundary] The bias $b_k(x)$ is such that for any $x\in \partial \ccalF$ we have that $b_k(x)=0$ for all $k$. 
\item[\bf Dependence with $k$] The scaled bias 
\begin{equation}
\tilde{b}_k(x) = b_k(x)\left(f_0(x)^k+\beta(x))^{1+1/k}\right)
\end{equation}
 is such that for any point $x$ in the interior of the free space $\ccalF$ we have that 
\begin{equation}
\|b_k(x)\| =O(1/k),
\end{equation} 
where $O(1/k)$ is a function satisfying $\lim_{k\to\infty}O(1/k)k = M$ with $M$ a positive constant.
\item[\bf Discontinuities away of the boundary] There exists a constant $D>0$ such that the function $b_k(x)$ is differentiable for all $x\in\ccalF$ satisfying $\beta_i(x)<D$ for every $i=1\ldots m$.
%
%
\item[\bf Regularity Assumption]
Let $\ccalU_k^{i}$ be the set defined as 
\begin{equation}
\begin{split}
\ccalU_k^{i}=\left\{x\in\ccalF\big| \nabla \varphi_k(x)^T\left(\nabla \varphi_k(x)+b_k(x)\right)\leq0 \right\} \\
\cap \left\{x\in\ccalF\big| \beta_i(x)\leq D \right\}.
\end{split}
\end{equation}
Since $\varphi_k(x)$ is a Morse function the vector field $\nabla \varphi_k(x)$ is strucutraly stable (c.f. Theorem 1.4 p.127 \cite{palis1970structural}). This is, there exists $\varepsilon_k>0$ such that for any function $g(x)$ satisfying $\|g(x)\|_{C^1}<\varepsilon_k$ we have that the orbits of $\dot{x} = \varphi_k(x)+g(x)$ are conjugate to those of $\dot{x}=\varphi_k(x)$. We assume the bias $b_k(x)$ be such that $\|b_k(x)\|<\varepsilon_k$ for any $x\in\ccalU_k^i$.
%
\end{mylist}
\end{assumption}
As discussed in Sections \ref{sec_circles} and \ref{sec_ellipse} the bias $b_k(x)$ accounts for a mismatch between the free space and the free space that the agent is able to estimate. This mismatch does not introduce a problem as long as the Regularity Assumption holds as we show in Section \ref{sec_biased}, where we show that despite this mismatch the agent is able to converge to a point that is arbitrarily close to the minimum of the objective function. However the Regularity Assumption limits the mismatch between the true environment and the model that the agent may have of it. In that sense, it is not clear to us whether this assumption is a limitation on the type of hallucinated obstacles that can be used to fit a given world or if it is a limitation on the analysis in Section \ref{sec_biased}. In the next section we present a preliminary result for unbiased estimates.
\section{Unbiased Estimator}\label{sec_unbiased}
In this section we consider the particular case of an agent that has access to an unbiased estimator of the gradient of the navigation function rather than the general model presented in \eqref{eqn_estimate_form}. This means that the bias is identically zero $b_k(x)\equiv 0$. The main result of this section is that an agent that follows the gradient update \eqref{eqn_gradient_descent} converges to the minimum of the navigation function $\varphi_k(x)$ defined in \eqref{eqn_navigation_function} while avoiding the obstacles with probability one. Therefore solving problem \eqref{eqn_navigation_problem}. We start by showing that the update proposed ensures obstacle avoidance. In the continuous time and deterministic framework this is a trivial consequence of the fact that the navigation function is admissible. Due to both the discretization  and the stochasticity this not longer the case unless the step size is small enough. The following lemma formalizes this result.
\begin{lemma}\label{lemma_non_collision}
Let $\mathcal{F}$ be the free space defined in \eqref{def_freespace} verifying Assumption \ref{as_obstacles}. Furthermore, let $\hat{g}_t(x_t,\theta_t)$ be an estimate of the gradient of the navigation function \eqref{eqn_navigation_function} satisfying Assumption \ref{assumption_estimator_model}. Then, by choosing a step size satisfying Assumption \ref{assumption_step_size} with $\varepsilon_0<\min_i\gamma_i/B$, where $\gamma_i$ and $B$ are defined in Assumption \ref{assumption_estimator_model}, the update \eqref{eqn_gradient_descent} is such that the sequence $\left\{x_t,t\geq0 \right\}\in \ccalF$.
\end{lemma}
\begin{proof}
Denote by $d_i(x)$ the euclidean distance of the point $x$ to the set $\ccalO_i$ and observe that by virtue of the triangular inequality one has that
\begin{equation}\label{eqn_non_collision_lemma}
d_i(x_{t+1}) \geq d_i(x_{t})-\varepsilon_t \|\hat{g}_t(x_t,\theta_t)\|. 
\end{equation}
Because the estimate of the gradient of the navigation function satisfies that $\|\hat{g}_t(x_t,\theta_t) \| \leq B$ (c.f. Assumption \ref{assumption_estimator_model}) and $\varepsilon_t$ is a decreasing sequence (c.f. Assumption \ref{assumption_step_size}), if $\varepsilon_0 \leq \min_i\gamma_i /B$ we have that $\varepsilon_t \|\hat{g}_t(x_t,\theta_t)\| < \min_i\left\{\gamma_i \right\}$. Therefore, for cases in which $d_i(x_t)\geq \gamma_i$ \eqref{eqn_non_collision_lemma} can be lower bounded by
\begin{equation}
d_i(x_{t+1}) > \gamma_i- \min_i\gamma_i \geq 0.
\end{equation}
The above implies that if at time $t$, the iterate $x_t$ is at a distance larger than $\gamma_i$ of the obstacle $\ccalO_i$ then at time $t+1$ the iterate $x_{t+1}$ remains in the free space. We are left to show that this is also true for cases where $d_i(x_t) <\gamma_i$. By Assumption \ref{assumption_estimator_model}, in this case we have that $-\hat{g}_t(x_t,\theta_t)^T \nabla \beta_i(x_t) >0$ and therefore non collision with obstacle $\ccalO_i$ is ensured trivially.
\end{proof}
%
The previous lemma shows that for a small enough initial step size the update \eqref{eqn_gradient_descent} is such that it avoids collisions. Observe that the previous result holds independently of the fact that the estimate is unbiased, so non collision is ensured both in the biased and unbiased cases. We next show that when the estimate is unbiased the gradient descent update \eqref{eqn_gradient_descent} converges almost surely to the set of critical points of the navigation function \eqref{eqn_navigation_function}. 
%
\begin{lemma}\label{lemma_convergence_to_critical_points}
Let $\mathcal{F}$ be the free space defined in \eqref{def_freespace} verifying Assumption \ref{as_obstacles} and let \eqref{eqn_condition_general} hold. Denote by $\hat{g}_t(x_t,\theta_t)$ an unbiased estimate of the gradient of the artificial potential \eqref{eqn_navigation_function} satisfying Assumption \ref{assumption_estimator_model} with $b(x)\equiv 0$. Furthermore, let $\varepsilon_t$ be a sequence satisfying Assumption \ref{assumption_step_size} with $\varepsilon_0<\min_i\gamma_i/ B$, where $\gamma_i$ and $B$ are defined in Assumption \ref{assumption_bias}. Then, there exists $K>0$ such that for any $x_0\in \ccalF$ and for any $k>K$ the sequence generated by the update \eqref{eqn_gradient_descent} is such that
\begin{equation}
\lim_{t\to \infty} x_t = X_c \quad \mbox{a.e.},
\end{equation}
where $X_c$ is a random variable taking values on the set of the critical points of $\varphi_k(x)$. 
\end{lemma}
\begin{proof}
By virtue of Theorem \ref{theo_navigation_function} there exists $K>0$ such that for any $k>0$ the function $\varphi_k(x)$ defined in \eqref{eqn_navigation_function} is a navigation function. Let us write $\varphi_k(x_{t+1})$ in terms of the previous iterate using the update rule given in \eqref{eqn_gradient_descent} and the Taylor expansion of $\varphi_k(x)$ around the point $x_t$
\begin{equation}\label{eqn_taylor_expansion}
\begin{split}
&\varphi_k(x_{t+1})= \varphi_k\left(x_t - \varepsilon_t \hat{g}_t(x_t,\theta_t)\right) =\\
& \varphi_k(x_t) - \varepsilon_t\nabla \varphi_k(x_t) \hat{g}_t(x_t,\theta_t) 
+\frac{\varepsilon_t^2}{2} \hat{g}_t(x_t)^T \nabla^2 \varphi_k (z)\hat{g}_t(x_t),
\end{split}
\end{equation}
where $z$ is a point in the segment $ x_t - \mu\varepsilon_t\hat{g}_t(x_t)$ with $\mu\in [0,1]$. Since the sequence of iterates is contained in the free space $\ccalF$ (c.f. Lemma \ref{lemma_non_collision}), so is $z$. The free space being  a compact set and $\varphi_k(x)$ being a twice differentiable function (c.f. Definition \ref{def_navigation_function}), the maximum eigenvalue of the Hessian of $\varphi_k(x)$ is upper bounded by a constant. Let $L$ be an upper bound for this eigenvalue. Then the quadratic term in \eqref{eqn_taylor_expansion} can be bounded as
\begin{equation}\label{eqn_trivial_bound}
 \hat{g}_t^T (x_t,\theta_t)\nabla^2 \varphi_k (z)\hat{g}_t(x_t,\theta_t) \leq  L \|\hat{g}_t(x_t)\|^2.
 \end{equation} 
Consider the expectation with respect to the sigma field $\ccalG_t$ on both sides of \eqref{eqn_taylor_expansion}. Using the linearity of the expectation, the fact that $\varphi_k(x_t)$ is $\ccalG_t$ measurable and the bound derived in \eqref{eqn_trivial_bound} we have that
\begin{equation}
\begin{split}
\E{\varphi_k(x_{t+1})\Big| \ccalG_t} &\leq \varphi_k(x_t) - \varepsilon_t \E{\nabla \varphi_k(x_t)^T \hat{g}_t(x_t,\theta_t)\Big| \ccalG_t} \\
&+\varepsilon_t^2\frac{L}{2}\E{ \|{g}_t(x_t,\theta_t) \|^2\Big| \ccalG_t}.
\end{split}
\end{equation}
Which by Assumption \ref{assumption_estimator_model} can be further upper bounded by
\begin{equation}\label{eqn_supermartingale1}
\begin{split}
\E{\varphi_k(x_{t+1})\Big| \ccalG_t} &\leq \varphi_k(x_t) - \varepsilon_t  \E{\nabla \varphi_k(x_t)^T\hat{g}_t(x_t,\theta_t)\Big| \ccalG_t} \\
&+\varepsilon_t^2\frac{LB^2}{2}.
\end{split}
\end{equation}
We next show that the following subsequence is a nonnegative supermartingale
\begin{equation}\label{eqn_supermartingale}
S_t= \varphi_k(x_{t}) +\sum_{s=t}^\infty \varepsilon_s^2 \frac{LB^2}{2}
\end{equation}
Since $\varphi_k(x)$ is a navigation function it is nonnegative and therefore $S_t$ is nonnegative sequence. Furthermore it is admissible and its value in the boundary is one, thus bounded. This fact in addition with the assumption that the selected step size $\varepsilon_t$ is a square summable sequence (c.f. Assumption \ref{assumption_step_size}) implies that $S_t$ is an integrable random variable. $S_t$ is also adapted to $\ccalG_t$ since $x_t$ is.  Thus, in order to show that $S_t$ is a nonnegative supermartingale it remains to be prooved that $\E{S_{t+1}\Big|\ccalG_t} \leq S_t$, which we do next. Using the linearity of the expectation and the bound for $\E{\varphi_k(x_{t+1})\Big| \ccalG_t} $ derived in \eqref{eqn_supermartingale1} we have that
\begin{equation}\label{eqn_half_supermartingale}
\begin{split}
\E{S_{t+1}\Big| \ccalG_t} &\leq \varphi_k(x_t)  +\sum_{s=t}^\infty \varepsilon_s^2\frac{LB^2}{2} \\
&-\varepsilon_t  \E{\nabla \varphi_k(x_t)^T\hat{g}_t(x_t,\theta_t)\Big| \ccalG_t}.
\end{split}
\end{equation}
Since we are considering an unbiased estimator satisfying \eqref{eqn_estimate_form}, we have that $\E{\hat{g}_t(x_t,\theta_t)\Big|\ccalG_t}=\alpha(x_t) \nabla \varphi_k(x_t)$ and therefore 
\begin{equation}
\E{\nabla \varphi_k(x_t)^T\hat{g}_t(x_t,\theta_t)\Big| \ccalG_t} = \alpha(x)\|\nabla \varphi_k(x_t)\|^2\geq 0
\end{equation}
since $\alpha(x)$ is strictly positive (c.f. Assumption \ref{assumption_estimator_model}). This completes the proof that $S_t$ is non negative supermartingale. Thus we have that (see e.g. Theorem 5.2.9 in \cite{durrett2010probability})%
\begin{equation}\label{eqn_convergence_to_random_variable}
\lim_{t\to\infty} S_t = S \quad \mbox{a.e.},
\end{equation}
where $S$ is a random variable such that $\E{S} \leq \E{S_0}$ and 
\begin{equation}
\sum_{t=0}^\infty \varepsilon_t \alpha(x_t)\|\nabla \varphi(x_t)\|^2 < \infty \quad \mbox{a.e.}.
\end{equation}
Since the sequence of step sizes $\left\{\varepsilon_t,t \geq 0\right\}$ is not summable and $\alpha(x)$ is bounded away from zero (c.f. Assumption \ref{assumption_estimator_model}) the convergence of the above series implies that 
\begin{equation}
\liminf_{t\to\infty} \|\nabla \varphi(x_t)\|^2  = 0 \quad \mbox{a.e.}.
\end{equation}
Therefore, there exists a subsequence $\{x_{t_s}, s\in\mathbb{N}\cup\{0\}\}$ that converges to the set of critical points of the navigation function  $\varphi_k(x)$. Since the limit of $S_t$ exists we have that
\begin{equation}
\lim_{s\to\infty} \varphi_k(x_{t_s}) = S \quad \mbox{a.e.}
\end{equation}
Moreover the critical points of the navigation function are hyperbolic (c.f. Definition \ref{def_navigation_function}), and therefore the limit of the sequence $x_t$ generated by the update \eqref{eqn_gradient_descent} is either the minimum of $\varphi_k(x)$ or one of the saddles of $\varphi_k(x)$. Thus completing the proof of the lemma.
\end{proof}
%
The previous lemma states that with probability one the update \eqref{eqn_gradient_descent} results in a sequence that converges to either the minimum of the navigation function $\varphi_k(x)$ or to one of its saddle points. In the deterministic and continuous time framework, the stable manifold of the saddles has zero measure and therefore, for a set of initial conditions of measure one we can guarantee convergence to its minimum. The next lemma is the analogous of this statement for the stochastic setting, where we show that the probability of converging to a saddle is zero. We state the result in its generic form for any hyperbolic function. 
%
%
\begin{lemma}\label{lemma_non_convergence_to_saddle}
Let $V(x): \ccalF \to \mathbb{R}$ be a hyperbolic function. Consider the sequence generated by the update of the form given in \eqref{eqn_gradient_descent} for which $\hat{g}_t(x_t,\theta_t)$ satisfies 
\begin{equation}\label{lemma_energy_function_eqn1}
\E{\hat{g}_t^T(x_t,\theta_t)\nabla V(x_t)\Big| \ccalG_t} > 0,
\end{equation}
if $x_t$ is not a critical point of $V(x)$ and
\begin{equation}\label{lemma_energy_function_eqn2}
\E{\hat{g}_t^T(x_t,\theta_t)\nabla V(x_t)\Big| \ccalG_t} = 0,
\end{equation}
if $x_t$ is a critical point of $V(x)$. Then for any $x_0 \in\ccalF$, the probability of the sequence $\left\{x_t, t\geq 0 \right\}$ converging to a saddle point of $V(x)$ is zero. 
\end{lemma}
\begin{proof}
See Section \ref{ap_non_convergence_to_saddle}
\end{proof}
%
%
As mentioned before, Lemma \ref{lemma_non_convergence_to_saddle} is more general than what is needed to show that the probability of converging to the saddle point of the navigation function is zero. In particular observe that by substituting $V(x)$ by $\varphi_k(x)$ and considering the case of an unbiased estimator the left hand side of \eqref{lemma_energy_function_eqn1} and \eqref{lemma_energy_function_eqn2} yields $\alpha(x_t)\|\varphi_k(x_t)\|^2$ which is strictly positive if $x_t$ is not a critical point of $\varphi_k(x)$ and is zero if $x_t$ is a critical point of $\varphi_k(x)$. Therefore in the particular case where we take $V(x)$ to be the navigation function $\varphi_k(x)$ and $\hat{g}_t(x_t,\theta_t)$ to be an unbiased estimator of the gradient of the navigation function the above lemma states that with probability zero the sequence $\left\{x_t \in \mathbb{R}^n, t \in \mathbb{N}\cup \left\{0\right\}\right\}$ given by the update \eqref{eqn_gradient_descent} converges to a saddle point of the navigation function $\varphi_k(x)$ for any initial position $x_0\in\ccalF$. Thus, by combining lemmas \ref{lemma_convergence_to_critical_points} and \ref{lemma_non_convergence_to_saddle} we can show convergence to the minimum of the navigation function with probability one. This is the subject of the following Theorem where we establish that an agent that has available an unbiased estimate of the gradient of the navigation function $\varphi_k(x)$ defined in \eqref{eqn_navigation_function} converges to $x^*$ if $f_0(x^*)=0$ or to a point that is arbitrarily close to the minimum of the objective function $x^*$ if $f_0(x^*)\neq 0$ with probability one. 
%
%
\begin{theorem}\label{theo_unbiased}
Let $\ccalF$ be the free space defined in \eqref{eqn_freespace} verifying Assumption \ref{as_obstacles} and let $f_0:\ccalX \to \mathbb{R}$ be a function satisfying Assumption \ref{assum_objective_function} with minimum at $x^*$. Consider the artificial potential $\varphi_k:\ccalF\to[0,1]$ defined in \eqref{eqn_navigation_function} and let $\hat{g}_t(x_t,\theta_t)$ be an unbiased estimate of $\nabla \varphi_k(x)$ satisfying Assumption \ref{assumption_estimator_model}. Also let \eqref{eqn_condition_general} hold for all $i=1\ldots m$. Let $\left\{x_t, t\geq 0\right\}$ be the sequence generated by the update \eqref{eqn_gradient_descent} with a step size satisfying Assumption \ref{assumption_step_size} and $\varepsilon_0 <\min_i \gamma_i/B$ with $\gamma$ and $B$ defined in Assumption \ref{assumption_estimator_model}. Then for every $\delta>0$, there exists a constant $K$ such that if $k>K$, we have that  $\left\{x_t, t\geq0 \right\} \in \ccalF$ and 
\begin{equation}
\lim_{t\to \infty} x_t = x^* \quad \mbox{a.e.},
\end{equation}
if $f_0(x^*)=0$, or
\begin{equation}
\lim_{t\to \infty} x_t = \bar{x} \quad \mbox{a.e.},
\end{equation}
when $f_0(x^*)\neq0$, where $\|\bar{x}-x^*\|<\delta$. 
\end{theorem}
\begin{proof}
From Theorem \ref{theo_navigation_function} it follows that for every $\delta>0$ there exists some $K>0$ such that for any $k>K$ the artificial potential $\varphi_k(x)$ is a navigation function with minimum at $\bar{x}$ satisfying $\|\bar{x}-x^*\|<\delta$ if $f_0(x^*)\neq 0$ and with minimum at $x^*$ otherwise. Then, the fact that the sequence $\left\{x_t, t\geq0\right\}\in\ccalF$ is a direct consequence of Lemma \ref{lemma_non_collision} and the convergence to the minimum of the navigation function is a consequence of lemmas \ref{lemma_convergence_to_critical_points} and \ref{lemma_non_convergence_to_saddle}.
\end{proof}
%
The previous theorem states that an agent who has access to an unbiased estimate of the gradient of a Rimon-Koditschek navigation function succeeds in navigating towards the minimum of the objective function or to a point that is arbitrarily close to it with probability one while remaining on the free space by selecting the tuning parameter $k$ large enough. In section \ref{sec_other_nf} we generalize this result to arbitrary spaces and suitable navigation functions. In the next section we generalize the result of Theorem \ref{theo_unbiased} to case where the estimate biased.
%
\section{Biased Estimator}\label{sec_biased}
In this section we generalize Theorem \ref{theo_unbiased} presented in Section \ref{sec_unbiased} for biased estimators satisfying Assumption \ref{assumption_estimator_model} and \ref{assumption_bias}. The main difference with the unbiased estimator is that the estimate $\hat{g}_t(x_t,\theta_t)$ is not a descent direction in expectation for the navigation function $\varphi_k(x)$. However it can be shown that there exists an energy like function that has the same structural properties as $\varphi_k(x)$ for which the estimate is a descent direction in expectation. We formalize this result in the next lemma. 
\begin{lemma}\label{lemma_energy_function}
Let $\ccalF$ be the free space defined in \eqref{eqn_freespace} verifying Assumption \ref{as_obstacles} and let $f_0:\ccalX \to \mathbb{R}$ be a function satisfying Assumption \ref{assum_objective_function} with minimum at $x^*$. Consider the artificial potential $\varphi_k:\ccalF\to[0,1]$ defined in \eqref{eqn_navigation_function} and let $\hat{g}_t(x_t,\theta_t)$ be an estimate of $\nabla \varphi_k(x)$ satisfying assumptions \ref{assumption_estimator_model} and \ref{assumption_bias}. Also let \eqref{eqn_condition_general} hold for all $i=1\ldots m$. Then, for every $\delta>0$ there is a constant $K$ such that if $k>K$, there exists a twice differentiable function $V_k: \ccalF \to \mathbb{R}$ whose critical points are at a distance smaller than $\delta$ to those of $\varphi_k(x)$. Furthermore, the index of the critical points of the two functions are equal and $V_k(x)$ is such that
\begin{equation}\label{lemma_energy_function_eqn1}
\E{\hat{g}_t^T(x_t,\theta_t)\nabla V_k(x_t)\Big| \ccalG_t} > 0,
\end{equation}
if $x_t$ is not a critical point of $V_k(x)$ and
\begin{equation}\label{lemma_energy_function_eqn2}
\E{\hat{g}_t^T(x_t,\theta_t)\nabla V_k(x_t)\Big| \ccalG_t} = 0,
\end{equation}
if $x_t$ is a critical point of $V_k(x)$. 
\end{lemma}
\begin{proof}
See Appendix \ref{ap_energy_function}. 
\end{proof}
%
%
In the above lemma we established the existence of an energy function for which the expected value of the estimate of the gradient of the navigation function  $\hat{g}_t(x_t,\theta_t)$ is a descent direction. In particular, the critical points of this energy function are arbitrarily close to those of the navigation function $\varphi_k(x)$. We are now in conditions of stating an proving the main result of the work, where we show that an agent that descends along the direction of a biased estimator of the gradient of a navigation function converges with probability one to a point that is arbitrarily close to the minimum of $f_0(x)$. We formalize this result next. 
%
%
\begin{theorem}\label{theo_biased}
Let $\ccalF$ be the free space defined in \eqref{eqn_freespace} verifying Assumption \ref{as_obstacles} and let $f_0:\ccalX \to \mathbb{R}$ be a function satisfying Assumption \ref{assum_objective_function} with minimum at $x^*$. Consider the artificial potential $\varphi_k:\ccalF\to[0,1]$ defined in \eqref{eqn_navigation_function} and let $\hat{g}_t(x_t,\theta_t)$ be an estimate of $\nabla \varphi_k(x)$ satisfying assumptions \ref{assumption_estimator_model} and \ref{assumption_bias}. Also let \eqref{eqn_condition_general} hold for all $i=1\ldots m$. Let $\left\{x_t, t\geq 0\right\}$ be the sequence generated by the update \eqref{eqn_gradient_descent} with a step size satisfying Assumption \ref{assumption_step_size} and $\varepsilon_0 <\min_i \gamma_i/B$ with $\gamma$ and $B$ defined in Assumption \ref{assumption_estimator_model}. Then for every $\delta>0$, there exists a constant $K$ such that if $k>K$, we have that  $\left\{x_t, t\geq0 \right\} \in \ccalF$ and 
\begin{equation}
\lim_{t\to \infty} x_t = \bar{x} \quad \mbox{a.e.},
\end{equation}
where $\bar{x}$ is a point arbitrarily close to $x^*$. 
\end{theorem}
\begin{proof}
Observe that non collision is ensured by virtue of Lemma \ref{lemma_non_collision}. Moreover because of Lemma \ref{lemma_energy_function} we know that there exists an energy function such that its critical points are arbitrarily close to those of $\varphi_k(x)$ and the indexes of said critical points are the same for both functions. Thus Lemma \ref{lemma_convergence_to_critical_points} holds for the self indexing energy function. Finally for $k$ large enough $\varphi_k(x)$ is a navigation function and thus Lemma \ref{lemma_non_convergence_to_saddle} and Theorem \ref{theo_navigation_function} hold completing the proof.
\end{proof}
%
The above theorem states that under the same conditions on the free space and the objective function than in the deterministic case, by following the update \eqref{eqn_gradient_descent} the agent is able, with probability one, to reach a point arbitrarily close to the minimum of the objective function $f_0(x)$ without running into the free space boundary. In particular, the update is performed by considering only local information about the objective function and the obstacles whereas in the construction in \cite{PaternainEtal15} (Theorem \ref{theo_navigation_function}) complete information about the obstacles is needed. Furthermore, instead of requiring exact information about both the objective function and the obstacles, stochastic measurements suffice to solve the problem of interest. Notice that in Theorem \ref{theo_biased} it is implicitly stated the need of satisfying condition \eqref{eqn_condition_general}. Thus for the stochastic case the same comments than in the deterministic case regarding the geometry of the free space and the condition number of the Hessian of the objective function are pertinent. This is, it is easier to navigate the free space when the obstacles and the level sets of the objective function are close to spheres.

Observe that the bias of the estimator accounts for a mismatch between the real free space and the one that is hallucinated by the agent. As explained in sections \ref{sec_circles} and \ref{sec_ellipse} there are three main components of this bias; the obstacles that cannot be measured since they are far away from the agent, the error introduces for assuming a specific model of the obstacles (circles or ellipses) and the error in the estimation of the parameters of the model. In that sense, the Regularity Assumption tells us that the perception that the agent has about the world is not that different from the real world when the configuration of the robot is in a neighborhood of the saddle points of the navigation function.

A difference between the results in Theorem \ref{theo_navigation_function} -- complete and deterministic -- and Theorems \ref{theo_unbiased} and \ref{theo_biased} -- local and stochastic -- is in the sense in which the navigation is almost surely. While in the deterministic case the navigation is almost surely in the sense that except for a set of initial positions of measure zero --the stable manifold of the saddle points of $\varphi_k(x)$ -- the solutions of the dynamical system $\dot{x}= -\nabla \varphi_k(x)$ converge to the minimum of the objective function; in the stochastic case the goal is achieved with probability one. This means, that for any initial position the probability of converging to minimum of $f_0(x)$ is one. Even when the initial position of the system is a saddle point of $\varphi_k(x)$. 
%
%
\section{Alternative Artificial Potentials}\label{sec_other_nf}
Throughout this paper we focused on navigation functions that are of the Rimon Koditschek form, however the results here presented can be generalized to larger classes of artificial potentials. We devote the current section to do so by considering the generic case of any navigation function for which it is possible to build an unbiased estimator of its gradient and for biased gradients of a potential where the obstacles are encoded by a logarithmic barrier. In Section \ref{sec_biased} we showed that under certain geometrical conditions of the free space and the objective function an agent is able to navigate towards to the minimum of the objective function --or to a point that is arbitrarely close-- with probability one while remaining in the free space if the agent has access to an unbiased estimate of the gradient of a Rimon-Koditschek navigation function (Theorem \ref{theo_biased}). We next generalize this result to any free space and suitable navigation functions as long as the estimate of its gradient is unbiased. This allows to consider different families of navigation functions that are suitable for other geometries of the free space e.g. harmonic functions to navigate topologically complex spaces \cite{loizou2011closed,loizou2012navigation}.
%
%
\begin{corollary}
Let $\ccalF$ be a free space and let $\varphi: \ccalF\to [0,1]$ be a navigation function (c.f. Definition \ref{def_navigation_function}) with minimum at the agent's goal $x^*$. Let $\hat{g}_t(x_t,\theta_t)$ be an unbiased estimate of the gradient of the navigation function satisfying Assumption \ref{assumption_estimator_model}. Then the update rule \eqref{eqn_gradient_descent} generates a sequence $\left\{x_t,t\geq0 \right\}\in \ccalF$ and such that $\lim_{t\to\infty} x_t =x^*$. 
\end{corollary}
\begin{proof}
The non collision proof is a direct consequence of \ref{lemma_non_collision} and the convergence to the minimum of the navigation function follows from lemmas \ref{lemma_convergence_to_critical_points} and \ref{lemma_non_convergence_to_saddle}. Observe that these do not depend on the specific form of the free space nor the navigation function selected. 
\end{proof}
The previous result generalizes Theorem \ref{theo_unbiased} for any space and suitable navigation function, meaning that following the sequence that arises from descending along the direction of an unbiased stochastic gradient succeeds in navigating towards the minimum of the objective function without running into the free space boundary. Next, we extend the result for biased estimates (c.f. Theorem \ref{theo_biased}) for a different class of artificial potentials, that of logarithmic barriers.
Inspired in the optimization literature we define the following barrier function
\begin{equation}\label{eqn_log_stoch}
\phi_k(x) = f_0(x)-\frac{1}{k}\log(\beta(x)).
\end{equation}
The previous potential is not a navigation function since it is not defined in the boundary and its image is not bounded between zero and one. However its supremum is at the boundary of the free space and we will show that all the critical points of the previous equation are non degenerate and it has a unique minimum. Differentiate \eqref{eqn_log_stoch} to get 
\begin{equation}
  \nabla \phi_k(x) =  \nabla f_0(x) - \frac{\nabla \beta(x)}{k\beta(x)}.
\end{equation}
Observe that the previous expression is similar to that of the direction of the gradient considered in \ref{eqn_estimate_navigation_direction}. In particular the same fundamental properties of the critical points hold, i.e., non degeneracy and polarity follow from analogous proofs to those in \cite{PaternainEtal15}. Since $\nabla \beta(x)$ is not zero in the boundary of the free space (see proof of Lemma 2 in \cite{PaternainEtal15}) the critical points can be pushed by increasing $k$ either arbitrarily close to the minimum of $f_0(x)$ or arbitrarily close to $\beta(x)$. In particular, the first one can be showed to be a unique local minima and the second ones to be saddles. Furthermore the eigenvalues of the Hessian of these critical points depend on $k$ with the same order as in the case of Rimon-Koditschek artificial potentials. In that sense if we consider the sensor model discussed in Section \ref{sec_sensor_model} the assumptions for the bias of the estimate of the gradient (\ref{assumption_estimator_model} and \ref{assumption_bias}) are reasonable. Hence by following the negative direction of the gradient of $ \phi_k(x))$ we converge to a point arbitrarily close to the minimum of $f_0(x)$. We state formally this theorem after defining the estimate of the descent direction current position $x_t$ and random vector $\theta_t$
\begin{equation}\label{eqn_log_stoch_estimate}
  \hat{g}(x_t,\theta) = \hat{\beta}(x_t,\theta_t)\hat{\nabla}f_0(x_t,\theta_t) - \frac{\hat{\nabla}\beta(x_t,\theta_t)}{k}.
\end{equation}
Observe that the above direction is the estimate of the gradient of $\phi_k(x)$ multiplied by $\beta(x)$, this has been done in order to avoid the norm of the estimate being large near the boundary of the free space.
\begin{theorem}\label{theo_log}
Let $\ccalF$ be the free space defined in \eqref{eqn_freespace} verifying Assumption \ref{as_obstacles} and let $f_0:\ccalX \to \mathbb{R}$ be a function satisfying Assumption \ref{assum_objective_function} with minimum at $x^*$. Consider the artificial potential $\phi_k:\ccalF\to \mathbb{R}$ defined in \eqref{eqn_log_stoch} and let $\hat{g}_t(x_t,\theta_t)$, the estimate defined in \eqref{eqn_log_stoch_estimate} satisfy the assumptions \ref{assumption_estimator_model} and \ref{assumption_bias}. Also let \eqref{eqn_condition_general} hold for all $i=1\ldots m$. Let $\left\{x_t, t\geq 0\right\}$ be the sequence generated by the update \eqref{eqn_gradient_descent} with a step size satisfying Assumption \ref{assumption_step_size} and $\varepsilon_0 <\min_i \gamma_i/B$ with $\gamma$ and $B$ defined in Assumption \ref{assumption_estimator_model}. Then for every $\delta>0$, there exists a constant $K$ such that if $k>K$, we have that  $\left\{x_t, t\geq0 \right\} \in \ccalF$ and 
\begin{equation}
\lim_{t\to \infty} x_t = \bar{x} \quad \mbox{a.e.},
\end{equation}
where $\bar{x}$ is a point arbitrarily close to $x^*$. 
\end{theorem}
\begin{proof}
  Observe that non collision is ensured by virtue of Lemma \ref{lemma_non_collision}. The fact that the critical points of $\phi_k(x)$ are non degenerate and that only one of them is a minimum and it can be pushed arbitrarily close to the minimum of $f_0(x)$ can be shown in the same way as Lemmas 2-6 in \cite{PaternainEtal15}. Hence by virtue of Lemma \ref{lemma_energy_function} there exists an energy function such that its critical points are arbitrarily close to those of $\phi_k(x)$ and the indexes of said critical points are the same for both functions. Thus Lemma \ref{lemma_convergence_to_critical_points} holds for the self indexing energy function. Moreover since all the critical points but one are non degenerate saddles for large enough $k$ and by virtue of  Lemma \ref{lemma_non_convergence_to_saddle} the theorem is proved.
\end{proof}
The previous results extends the result for the biased estimate of the Rimon-Koditschek navigation function to a new class of artificial potentials under the same conditions over the geometry of the free space and the bias. In the next section we study the results of Theorems \ref{theo_biased} and \ref{theo_log} numerically. 
%
\section{Numerical Examples}\label{sec_numerical_example}
We evaluate the performance of the local stochastic approximation of the gradient of the navigation function given in \eqref{eqn_estimate_navigation_direction} in two different scenarios for which the condition \eqref{eqn_condition_general} is satisfied. In particular, the estimations of the obstacles are done by considering osculating circles at the closest point of the obstacle to the agent as in Section \ref{sec_circles}. In Section \ref{sec_elliptical_obstacles} the free space is such that the obstacles are ellipsoids and in section \ref{sec_egg_shaped} these are egg shaped. In both cases the external boundary of the free space is a spherical shell of center $c_0$ and radius $r_0$.
\subsection{Elliptical obstacles}\label{sec_elliptical_obstacles}
\begin{figure*}
        \centering
        \begin{subfigure}[b]{0.49\linewidth}
                \includegraphics[width=\linewidth, height=0.62\linewidth]{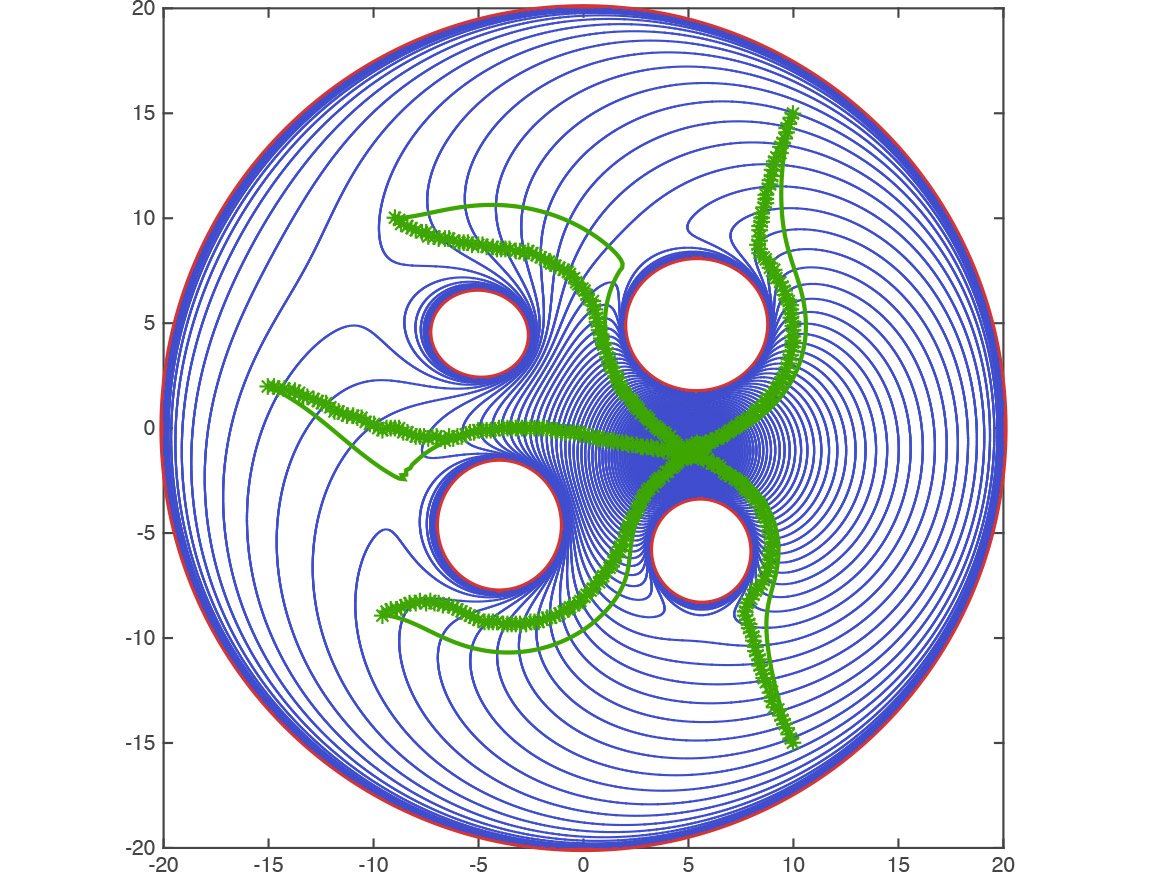}
                \caption{Trajectories resulting of the navigation function approach -- solid line-- and its stochastic approximation given in \eqref{eqn_gradient_descent} --stars-- for $k=7$.}
                \label{fig_k_7}
        \end{subfigure}
        ~
        \begin{subfigure}[b]{0.49\linewidth}
                \includegraphics[width=0.8\linewidth, height=0.62\linewidth]{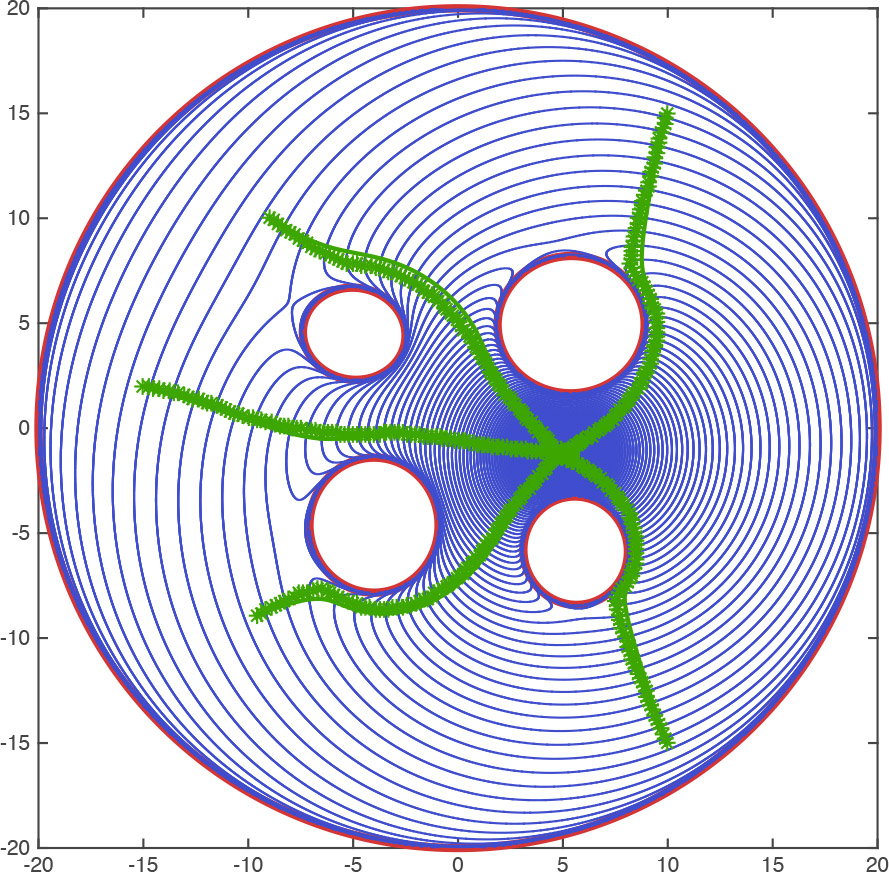}
	\caption{Trajectories resulting of the navigation function approach -- solid line-- and its stochastic approximation given in \eqref{eqn_gradient_descent}--stars-- for $k=12$.}
                \label{fig_k_12}
        \end{subfigure}
        ~ \caption{The trajectories resulting from the update \eqref{eqn_gradient_descent} succeed in driving the agent to the goal configuration for five different initial positions as expected in virtue of Theorem \ref{theo_biased}. We observe that the larger the order parameter $k$ is, the closer the trajectory resulting from stochastic approximation is to the trajectory resulting of descending along the gradient of the navigation function \eqref{eqn_navigation_function}.}\label{fig_ellipses}
\end{figure*}
%

\begin{figure*}
        \centering
        \begin{subfigure}[b]{0.49\linewidth}
                \includegraphics[width=\linewidth, height=0.62\linewidth]{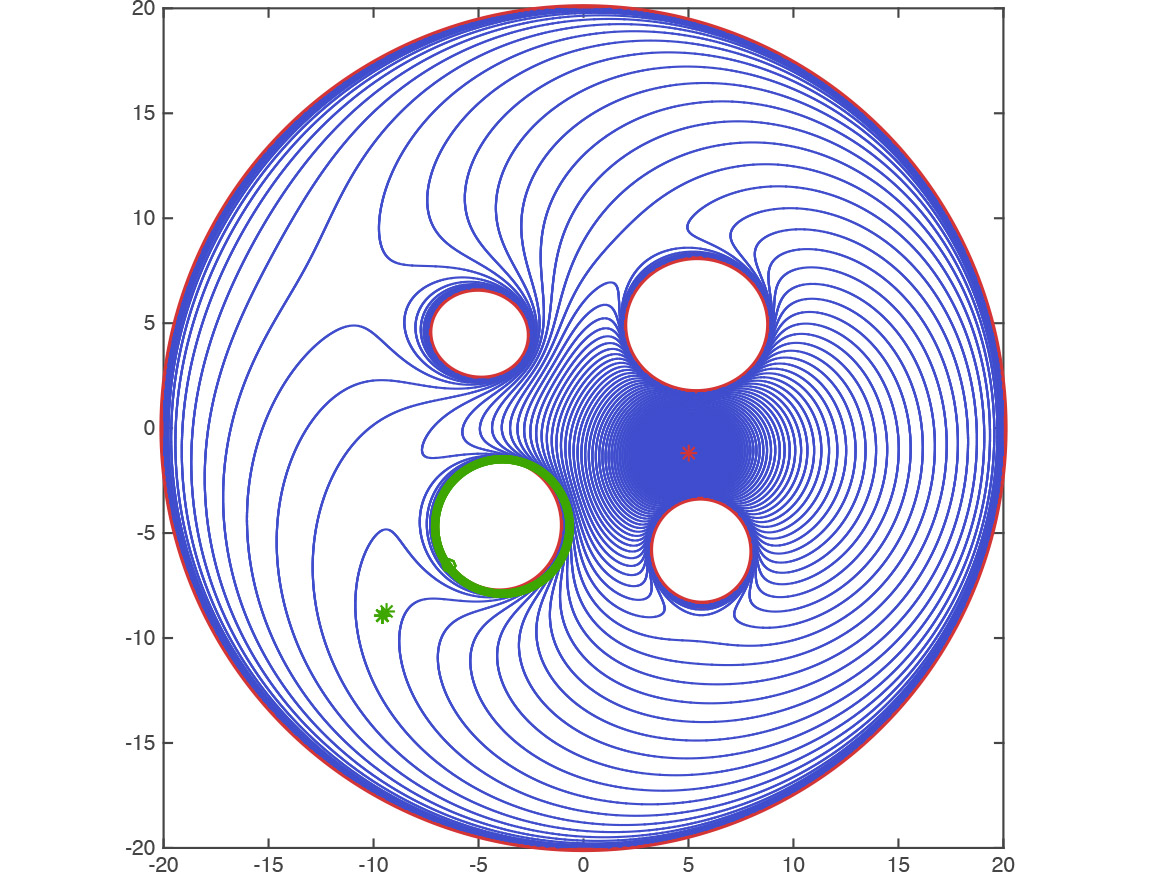}
                \caption{Local estimation of the obstacle with perfect measures.}
                \label{fig_no_noise}
        \end{subfigure}
        ~
        \begin{subfigure}[b]{0.49\linewidth}
                \includegraphics[width=\linewidth, height=0.62\linewidth]{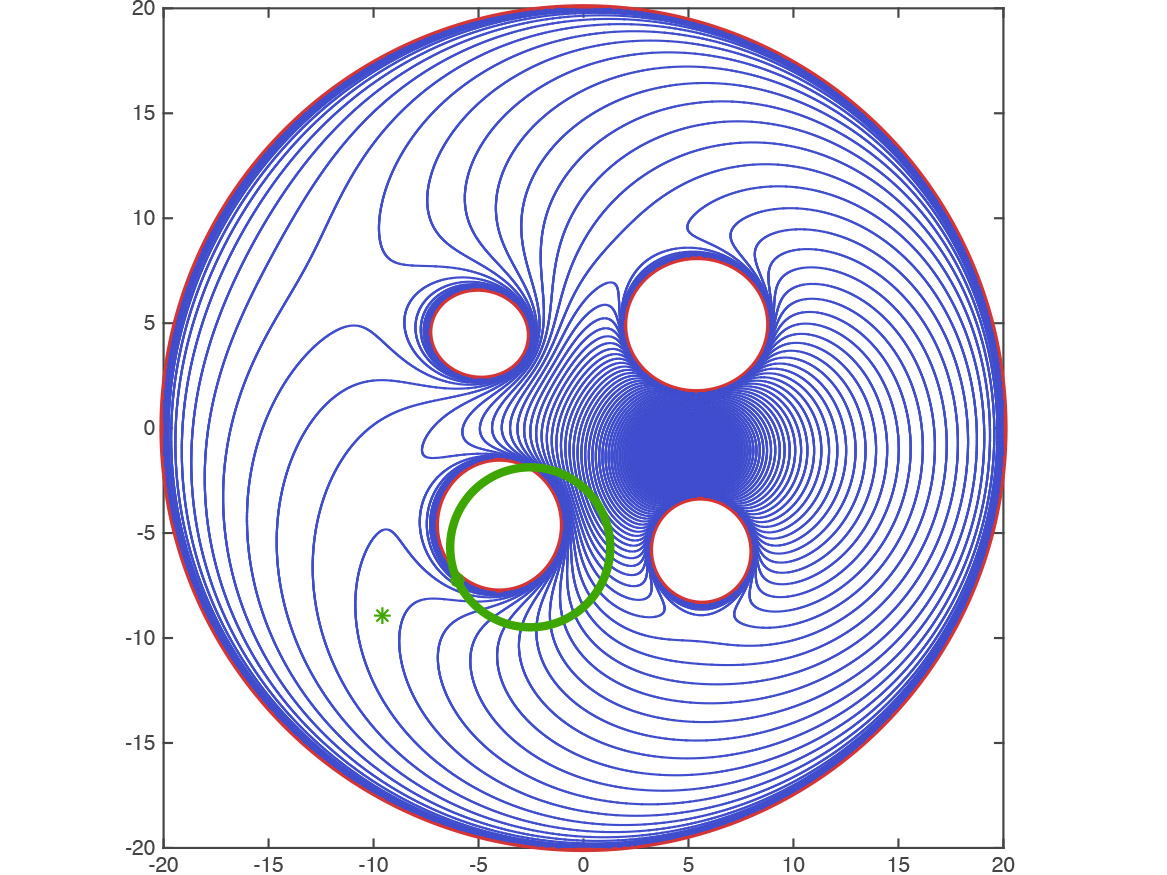}
                \caption{Stochastic estimation of the obstacle with noisy measurements. }
                \label{fig_noise}
        \end{subfigure}
        ~ \caption{Estimation of the obstacles by the hallucinated osculating circle for a particular position in the free space with exact and stochastic information. Obstacles are sensed if $d_i(x)<7$. Noise is Gaussian, additive, mean zero and with variance $\sigma_{d_i}=\sigma_{R_i}=\sigma_{\bbn_i}=d_i(x)/10$.  }\label{fig_obstacle_estimation}
\end{figure*}
%
%
In this section we consider $m$ elliptical obstacles in $\mathbb{R}^2$. For $i=1\ldots m$, let $A_i \in \ccalM^{2\times 2}$ be symmetric  and positive definite matrices, and let $\mu_{\min}^i>0$ be the minimum eigenvalue of matrix $A_i$. We describe the obstacles in a functional form through the following functions 
\begin{equation}
\beta_i(x) = (x-c_i)^T A_i (x-c_i) - \mu_{\min}^i r_i^2.
\end{equation}
where $c_i \in \ccalX$ is the center of the $i$-th ellipse and $r_i>0$ is the length of its largest axis. With this selection of $\beta_i(x)$ the $i$-th obstacle is defined as 
\begin{equation}
\ccalO_i = \left\{x \in \ccalX \big|  \beta_i(x) < 0\right\}.
\end{equation}
In these experiments we place the center of each ellipsoid in a different orthant. In particular, each center is set to be in the position $L(\pm1,\pm1)$ and then we add a random variation drawn uniformly from $[-\Delta,\Delta]^2$, where $0 < \Delta < L$. The maximum axis of the ellipse -- $r_i $-- is drawn uniformly from $[r_0/10,r_0/5]$ and the matrices $A_i$ for $i = 1 . . . m$ are such that they are orthogonal and their eigenvalues are random and uniformly selected from the interval $[1, 2]$. We verify that the obstacles resulting of the previous process do not intersect. If they do, we re draw all previous parameters. For the objective function we consider a quadratic cost given by $f_0(x)=(x-x^*)^T Q(x-x^*)$, where $x^*$ is drawn uniformly over $[-r_0 /2, r_0 /2]^2$ and we verify that it is in the free space. The matrix $Q \in \ccalM^{2\times 2}$ is a random positive definite symmetric matrix whose eigenvalues are selected as follows. For each obstacle we compute the maximum condition number that $Q$ could have in order to satisfy condition \eqref{eqn_condition_general}. Let $N_{cond}$ be the maximum among these admissible condition numbers. Then, the eigenvalues of $Q$ are selected randomly from $[1, N_{cond} + 1]$, hence ensuring that \eqref{eqn_condition_general} is satisfied. For the estimates of the objective function, its gradient, the distance to the obstacles, the normal direction to them and their curvature we consider independent gaussian additive noise with mean zero and standard deviation $\sigma_{q}$.  The step size selected for the update \eqref{eqn_gradient_descent} is of the form $\varepsilon_t = \varepsilon_0/(1+\zeta t)$ and the initial position is selected randomly over $[-r_0,r_0]^2$. 

For this experiment we set the parameters to be $c_0=0$, $r_0 = 20$, $L=6$, $\Delta = 1$, $\sigma_{f_0} = \sigma_{\nabla f_0} =1$ and $\sigma_{d_i} = \sigma_{R_i} = \sigma_{\bbn_i} = d_i(x)/10$. The selection of a variance that depends on the the distance is done so to ensure that the closer the agent is to the boundary of the free space the better the estimation of the obstacle is. In particular, at the boundary we have that $\sigma_{d_i} = \sigma_{R_i} = \sigma_{\bbn_i}=0$. We set the constant at which the agent is able to measure an obstacle [c.f. \eqref{eqn_awareness_set}] to be $c=7$. Finally, the parameters of the step size are $\varepsilon_0 = 5\times 10^{-2}$ and $\zeta =5\times 10^{-3}$ and we run each simulation $100$ steps. 

In Figure \ref{fig_ellipses} we observe the behavior of the system that follows the local and stochastic update \eqref{eqn_gradient_descent} -- marked with stars -- and that of the system following the gradient dynamical system $\dot{x}=-\nabla \varphi_k(x)$-- solid lines -- for five different initial conditions. In Figure \ref{fig_k_7} the order parameter is set to be $k=7$ while in \ref{fig_k_12} it is set to be $12$. In both cases it can be observed that the local and stochastic update succeeds in generating a sequence that remains in the free space and that converges to the minimum of the objective function. It is also observed that the direction in which the agent moves while following the local update differs from that of the agent following the gradient of the navigation function. This result is not surprising in virtue of the fact that as discussed in Section \ref{sec_circles} the model selected results in a biased estimate of the gradient of the navigation function. 

However notice that by increasing $k$ the two trajectories become closer to each other. This effect can be observed by comparing the trajectories depicted in figures \ref{fig_k_7} and \ref{fig_k_12} where the order parameter $k$ is set to be $7$ and  $12$ respectively. This result is expected because as discussed in Section \ref{sec_circles} the bias is such that its norm is decreasing with $k$. In particular by selecting $k$ large enough the bias could be reduced arbitrarily. Notice that when the order parameter $k$ is increased the sequence resulting from the stochastic approximation is not modified as much as the trajectory that considers complete information about the free space. This is because the larger the value of $k$ the smaller is the effect of obstacles that are far from the agent as compared to the gradient of the objective function (c.f. \eqref{eqn_navigation_function}). Thus in a sense higher value of $k$ resembles to considering only nearby obstacles as in the case of the stochastic approximation. 

The effect of the standard deviations of the noise in the estimation of the obstacles with which the simulations were done is illustrated in Figure \ref{fig_obstacle_estimation} by the green circles depicted. In particular, for the initial position of one of the trajectories depicted in Figure \ref{fig_k_7} we observe the estimation of the closest obstacle to that position in the noiseless case \ref{fig_no_noise} and the estimate with noise \ref{fig_noise}.
%
%
\subsection{Egg shaped world obstacles}\label{sec_egg_shaped}
In this section we consider egg shaped obstacles as an example of convex obstacles different than ellipses. We draw the center of the each obstacle, $c_i$, from a uniform distribution over $[-L/2, L/2] \times [-L/2, L/2]$. The distance between the ''tip'' and the ''bottom'' of the egg, $r_i$, is drawn uniformly over $[r_0/10; r_0/5]$ and with equal probability the egg is horizontal or vertical. The obstacle being horizontal translates into the fact that the function $\beta_i(x)$ representing the obstacle takes the following form
\begin{equation}\label{eqn_egg_hor}
\beta_i(x) = \|x-c_i \|^4 - 2r_i \left(  x^{(1)} - c_i^{(1)}\right)^3,
\end{equation}
where the superscript $(1)$ refers to first component of a vector. Likewise, for vertical eggs the function $\beta_i(x)$ takes the form
\begin{equation}\label{eqn_egg_ver}
\beta_i(x) = \|x-c_i \|^4 - 2r_i \left(  x^{(2)} - x_c^{(2)}\right)^3.
\end{equation}
Notice that the functions $\beta_i$ as defined above are not convex on $\mathbb{R}^2$, however since their Hessians are positive definite outside the obstacles it is possible to define a convex extension of them inside the obstacles. This is not needed because the agent operates in the free space and therefore there is no difference to him between the functions defined in \eqref{eqn_egg_hor} and \eqref{eqn_egg_ver} and their convex extensions. In particular, for this experiment we set $r_0=20$ and $L=6$ The selection of the noises standard deviations $\sigma_q$ and the distance at which the obstacles can be measured are the same as in Section \ref{sec_elliptical_obstacles}.

In Figure \ref{fig_egg} we observe the level sets of the navigation function \eqref{eqn_navigation_function} and the trajectories resulting from the stochastic approximation \eqref{eqn_gradient_descent} --marked with stars-- and from descending along the direction of the negative gradient of the navigation function for $k=15$. It can be observed that the update \eqref{eqn_gradient_descent} succeeds in driving the agent to the goal configuration given by the minimum of the objective function $f_0(x)$ while remaining in the free space at all times. 
%
\begin{figure}
\centering
\includegraphics[width=\linewidth, height=0.62\linewidth]{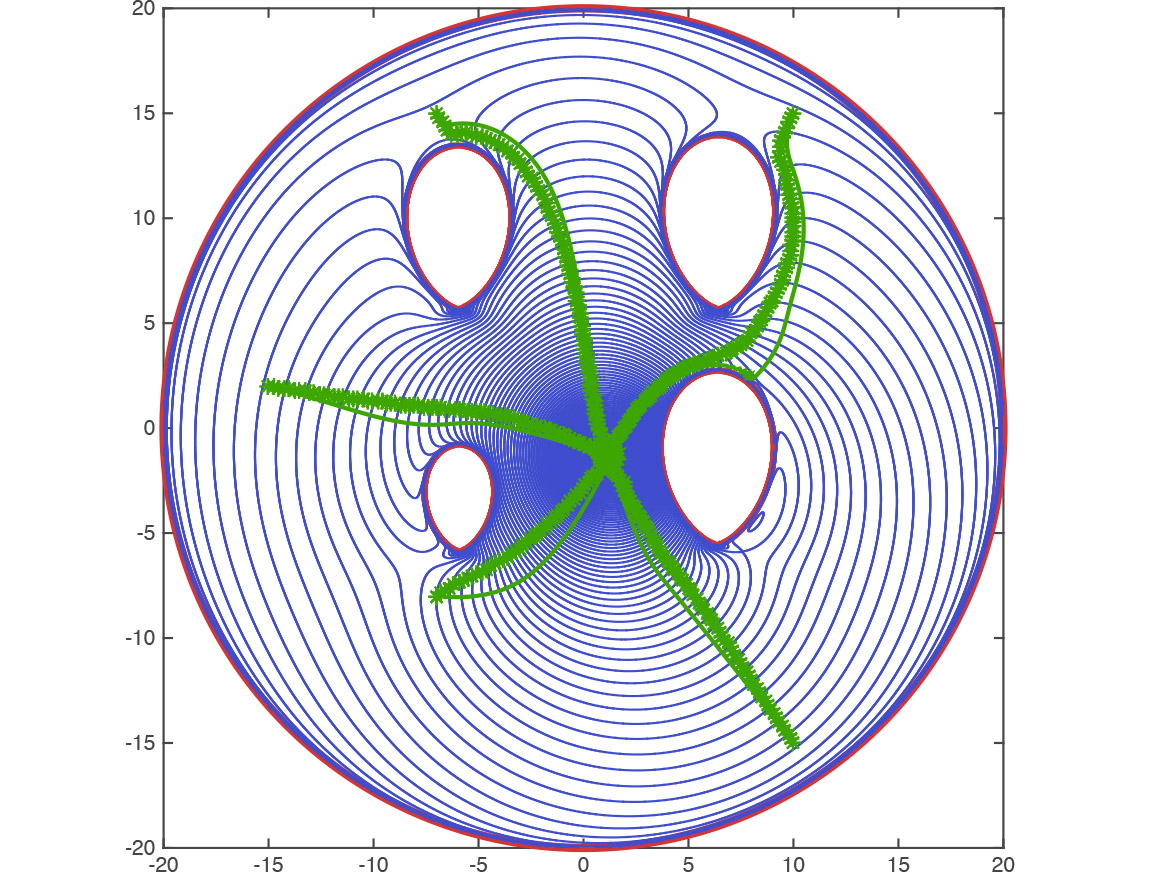}
\caption{Trajectories resulting of the navigation function approach -- solid line-- and its stochastic approximation given in \eqref{eqn_gradient_descent} for $k=15$ in an egg shaped world. The trajectories resulting from the update \eqref{eqn_gradient_descent} succeed in driving the agent to the goal configuration for five different initial positions as expected in virtue of Theorem \ref{theo_biased}. }\label{fig_egg}
\end{figure}
\subsection{Logarithmic barrier}
In this section we evaluate the performance of the descent along the direction of the negative gradient of the logarithmic barrier artificial potential in \eqref{eqn_log_stoch_estimate}. For this experiments the obstacles and the boundary of the workspace are selected as in Section \ref{sec_elliptical_obstacles} and the parameters selected are set to $c_0=0$, $r_0 = 20$, $L=6$, $\Delta = 1$, $\sigma_{f_0} = \sigma_{\nabla f_0} =1$, $\sigma_{d_i} = \sigma_{R_i} = \sigma_{\bbn_i} = d_i(x)/10$ and $k=10$. In Figure \ref{fig_log} we depict the trajectory of an agent starting at different initial positions. As it can be observed the agent succeeds in reaching the minimum of the objective function $f_0(x)$ while avoiding the obstacles. By comparing these trajectories to those in figures \ref{fig_k_7} and \ref{fig_k_7} --which were generated by following the gradient of the Rimon-Koditschek artificial potential-- we observe that the logarithmic barrier artificial potential results in paths that pass closer to the obstacles.  
\begin{figure}
  \centering
\includegraphics[width=0.85\linewidth]{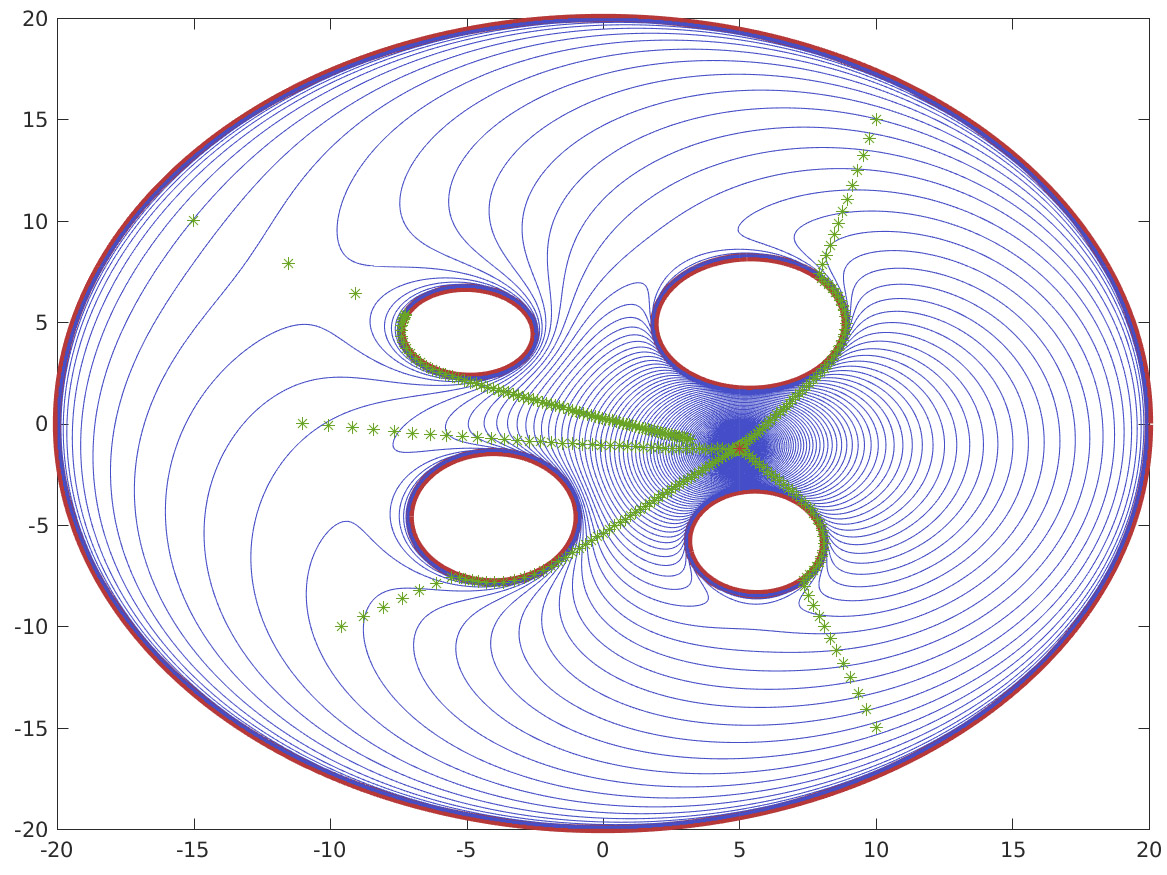}
\caption{Trajectories resulting of following the negative gradient of the logarithmic barrier given in \eqref{eqn_log_stoch} for $k=10$ in an elliptical world. The trajectories resulting from the update \eqref{eqn_gradient_descent} succeed in driving the agent to the goal configuration for five different initial positions as expected in virtue of Theorem \ref{theo_log}. }\label{fig_log}
\end{figure}

%
\section{Conclusions}\label{sec_conclusion}
We considered a set with convex holes in which an agent must navigate to the minimum of a convex function. The objective function and the obstacles are unknown a priori to the agent and sensorial information about these is available to him. In particular, this information is local and stochastic. We showed that an agent that is capable of constructing an unbiased estimate of the gradient of an artificial potential of the Rimon-Koditschek form is capable of navigating towards the minimum of this objective function while avoiding the obstacles with probability one under the same geometric restrictions than in the deterministic case. Furthermore, for biased estimates we show that if near the saddle points of the navigation function the bias is not too large the same holds true. Numerical experiments support the theoretical results.
%
\appendix
\section{Appendix}

\subsection{Proof of Lemma \ref{lemma_non_convergence_to_saddle}}\label{ap_non_convergence_to_saddle}
Let us add and subtract $\varepsilon_t\E{\hat{g}_t(x_t,\theta_t) \Big| \ccalG_t}$ to \eqref{eqn_gradient_descent} 
\begin{equation}
\begin{split}
x_{t+1} &= x_t -\varepsilon_t\left(\hat{g}_t(x_t,\theta_t)-\E{\hat{g}_t(x_t,\theta_t) \Big| \ccalG_t}\right) \\
&-\varepsilon_t\E{\hat{g}_t(x_t,\theta_t) \Big| \ccalG_t}.
\end{split}
\end{equation} 
Since $\hat{g}_t(x_t,\theta)$ is an unbiased estimator of the gradient of the function $V(x)$ we can think of the expression $\left(\hat{g}_t(x_t,\theta_t)-\E{\hat{g}_t(x_t,\theta_t) \Big| \ccalG_t}\right)$ as an error $e_t$ between the stochastic gradient and the gradient of the function $V(x)$. With this definition the above equation can be written as
\begin{equation}\label{eqn_equation_aux_1_lemma_saddle}
x_{t+1} = x_t -\varepsilon_t\nabla V(x_t) -\varepsilon_te_t, 
\end{equation}
where $e_t$ is a random vector whose expected value is zero and it is bounded with probability one because $\hat{g}_t(x_t,\theta_t)$ is bounded with probability one. Let $x_c$ be a saddle point of the energy function $V(x)$ and let $H$ denote the Hessian of $V(x)$ evaluated at $x_c$, i.e., $H= \nabla^2 V(x_c)$. Then, we have that $\nabla V(x_t) = H(x_t-x_c) + o(\|x_t-x_c\|^2)$. Replacing this expression for the gradient of $\nabla V(x_t)$ in \eqref{eqn_equation_aux_1_lemma_saddle} yields,
\begin{equation}\label{eqn_key_in_saddle_point_lemma}
x_{t+1} -x_c= (I-\varepsilon_tH)(x_t -x_c) +\varepsilon_t\left(o(\|x_t-x_c\|^2) -e_t\right).
\end{equation}
Recursively it is possible to write the difference $x_{t+1} - x_c$ as
\begin{equation}
\begin{split}
x_{t+1} - x_c &= \prod_{s=0}^t (I-\varepsilon_sH) (x_0-x_c)\\
&+\sum_{s=0}^t\varepsilon_s\left(\prod_{u=s}^{t-1} I -\varepsilon_u  H\right)\left( o(\|x_s-x_{c}\|^2) -e_s\right).
\end{split}
\end{equation}
Let $v_i$ be the eigenvector corresponding to the eigenvalue $\lambda_i$ of the Hessian, then we can write the projection over $v_i$ of the above equation as
\begin{equation}
\begin{split}
(x_{t+1} - x_c)_{i} &= \prod_{s=0}^t (1-\varepsilon_s\lambda_i) (x_0-x_c)_{i}\\
&+\sum_{s=0}^t\varepsilon_s\left( o(\|x_s-x_{c}\|^2) -e_s\right)_{i}\prod_{u=s}^{t-1} \left(1 -\varepsilon_u  \lambda_i\right).
\end{split}
\end{equation}
Taking $\prod_{s=0}^{t-1} (I-\varepsilon_s\lambda_i) $ as a common factor we can write the above equation as
\begin{equation}\label{eqn_key_equation}
\begin{split}
(x_{t+1} - x_c)_{i} = \prod_{s=0}^{t-1} (1-\varepsilon_s\lambda_i) \left[(1-\varepsilon_t\lambda_i)(x_0-x_c)_{i}+ \phantom{\sum_{s=0}^t}\right.\\
\left.\sum_{s=0}^t\varepsilon_s\left(\prod_{u=0}^{s-1} (1 -\varepsilon_u  \lambda_i)^{-1}\right)\left( o(\|x_s-x_{c}\|^2) -e_s\right)_{i}\right].
\end{split}
\end{equation}
Let us assume that the sequence resulting from the update given in \eqref{eqn_gradient_descent} converges to a saddle point with strictly positive probability. Therefore, there is a subset of $\Omega$ for which for any $\delta>0$, there exists a time $T$ such that the absolute value of the sequence $x_{t+1}-x_c$ is smaller than $\delta$ for any $t>T$. Without loss of generality let $T=0$. This implies that for every $s\geq0$ we have that $o(\|x_s-x_c\|^2)$ is uniformly bounded. Next, we will show that the series $\sum_{s=0}^t\varepsilon_s\left(\prod_{u=0}^{s-1} (1 -\varepsilon_u  \lambda_i)^{-1}\right)$ converges. Let us start by writing $\prod_{u=0}^{s-1} (1 -\varepsilon_u  \lambda_i)^{-1}$ as
\begin{equation}
\prod_{u=0}^{s-1} (1 -\varepsilon_u  \lambda_i)^{-1} = \prod_{u=0}^{s-1} \frac{1+\zeta  u}{1-\varepsilon_0\lambda_i + \zeta u}
\end{equation}
Divide both numerator and denominator by $\zeta$ and write the quotient of products as the following quotient of gamma functions 
\begin{equation}\label{eqn_aux_product}
\prod_{u=0}^{s-1} (1 -\varepsilon_u  \lambda_i)^{-1} = \frac{\Gamma(1/\zeta+s)}{\Gamma((1-\varepsilon_0\lambda_i)/\zeta+s)}\frac{\Gamma((1-\varepsilon_0\lambda_i)/\zeta)}{\Gamma(1/\zeta)}.
\end{equation}
Let $s$ tend to infinity and write the limit of the gamma function evaluated in $c+s$ for any $c$ as
\begin{equation}
\lim_{s\to \infty} \Gamma(c+s) = \lim_{s\to\infty} \Gamma(s)s^c. 
\end{equation}
Therefore the limit of the expression \eqref{eqn_aux_product} for $s$ tending to infinity can be computed using the asymptotic behavior of the gamma function from the above equation. This limit yields
\begin{equation}
\lim_{s\to\infty}\prod_{u=0}^{s-1} (1 -\varepsilon_u  \lambda_i)^{-1} =\frac{\Gamma((1-\varepsilon_0\lambda_i)/\zeta)}{\Gamma(1/\zeta)}s^{\varepsilon_0\lambda_i/\zeta}
\end{equation}
Since the index of the critical point $x_c$ is $n-1$, we have $n-1$ eigenvalues that are strictly negative. For any of these we have that the asymptotical behavior of $\varepsilon_s\left(\prod_{u=0}^{s-1} (1 -\varepsilon_u  \lambda_i)^{-1}\right) $ is $o(s^{-q})$, with $q>1$ and therefore
\begin{equation}
\sum_{s=0}^\infty\varepsilon_s\left(\prod_{u=0}^{s-1} (1 -\varepsilon_u  \lambda_i)^{-1}\right) < \infty.
\end{equation}
This implies in turn that \eqref{eqn_key_equation} can be written as 
\begin{equation}\label{eqn_limit_saddle}
\lim_{t\to \infty} (x_{t+1} - x_c)_{i} = \lim_{t\to\infty} \prod_{s=0}^t (1-\varepsilon_s\lambda_i) \left[(x_0-x_c)_{i} +C\right],
\end{equation}
where $C$ is given as
\begin{equation}
C=\sum_{s=0}^\infty\varepsilon_s\left(\prod_{u=0}^{s-1} (1 -\varepsilon_u  \lambda_i)^{-1}\right)\left( o(\|x_s-x_{c}\|^2) -e_s\right)_{i}.
\end{equation}
Without loss of generality we can assume that $(x_0-x_c)_{i}$ it is not zero, because in finite time with probability one any component of the update will be different than zero. In the subset of the probability space for which $\lim_{t\to \infty} x_t(\omega) = x_c$, the left hand side of \eqref{eqn_limit_saddle} is equal to zero. However, the right hand side of \eqref{eqn_limit_saddle} diverges since $\lambda_i<0$ which is a contradiction. In fact, in order to ensure divergence of the right hand side of \eqref{eqn_limit_saddle} we need to show that $C = (x_0 -x_c)_i$ only in a set of zero measure. Since we are assuming that $\lim_{t\to\infty} x_t =x_c$ the approximation errors $o(\| x_s- x_c\|)$ are arbitrarily small. Thus, in order to have $C = (x_0-x_c)_i$ it must be the case that the sum of independent random variables $(e_s)_i$ weighted by its corresponding coefficients is equal to $(x_0-x_c)_i$. Which cannot hold since these are independent of the initial position. Thus, the set for which $\lim_{t\to \infty} x_t(\omega) = x_c$ has measure zero. Thus completing the proof of the Lemma.


\subsection{Proof of Lemma \ref{lemma_energy_function}}\label{ap_energy_function}
To develop the proof of Lemma \ref{lemma_energy_function} we need the definition of a gradient like vector field and a theorem by Smale that states that any gradient like vector field on a manifold has a self indexing energy function \cite{10.2307/1970311}. We formalize this result next after providing the definition of a gradient like vector field.
\begin{definition}[\bf Gradient like vector field]\label{def_gradient_like_field}
Let $x \in \mathbb{R}^n$ and let $g : \mathbb{R}^n \to \mathbb{R}^n$ be a smooth function, we say that $g(x)$ is a gradient like vector field if its non wandering set consists of finitely many hyperbolic equilibrium states and the stable and unstable manifolds of singular points intersect transversally. 
\end{definition}
Instead of presenting the original version of Smale's Theorem in \cite{10.2307/1970311} we provide a more recent version of it that can be found in \cite{grines2014energy}.
\begin{theorem}\label{theo_energy_function_smale}
Let $M^n$ be a smooth closed orientable manifold and let $g(x): M^n \to [0,n]$ be a gradient-like vector field, then, there exists a function $V: M^n \to \mathbb{R}$ such that 
\begin{description}
\item[(i)] is twice differentiable and all of its critical points are nondegenerate,
\item[(ii)] its critical points coincide with the set of the critical points of $g(x)$
\item[(iii)] $\dot{V}(x) = \nabla V(x)^T g(x) < 0$, for any $x$ such that $g(x) \neq 0$
\item[(iv)] V(x) = ind(x) for $x$ such that $ g(x) =0$.
\end{description}
\end{theorem}
\begin{proof}
See Theorem B in \cite{10.2307/1970311}.
\end{proof}
In virtue of the previous theorem to prove the existence of a function $V_k(x)$ satisfying \eqref{lemma_energy_function_eqn1} and \eqref{lemma_energy_function_eqn2} it suffices to show that the vector field $\nabla \varphi_k(x)+b_k(x)$ is gradient-like. This however is not possible since $b_k(x)$ is not differentiable but piece-wise differentiable (c.f. Assumption \ref{assumption_bias}). We consider then a smooth approximation $b^{diff}_k(x)$ of the $b_k(x)$ and show that $\nabla \varphi_k(x)+b^{diff}_k(x)$ is gradient like. We formalize this result in the next lemma thus showing that a self indexing function for the smooth approximation of the vector field of interest exists. 
\begin{lemma}\label{lemma_gradient_like}
Let $\ccalF$ be the free space defined in \eqref{eqn_freespace} verifying Assumption \ref{as_obstacles} and let $\varphi_k:\ccalF\to[0,1]$ be the function defined in \eqref{eqn_navigation_function}. Let $\lambda_{\max}, \lambda_{\min}$ be the bounds from Assumption \ref{assum_objective_function} and $\mu_{\min}^i$ be the minimum eigenvalue of the Hessian of $\beta_i(x)$. Furthermore let \eqref{eqn_condition_general} hold for all $i=1\ldots m$ and let $b_k(x)$ satisfy Assumption \ref{assumption_bias}. Define a smooth approximation $b^{diff}_k(x)$ of $b_k(x)$, then there exists a constant $K$ such that if $k>K$, the vector field $\nabla \varphi_k(x)+b^{diff}_k(x)$ is gradient like.
\end{lemma}
\begin{proof}
Observe that $\nabla \varphi_k(x)$ and $b_k(x)$ share a commun factor $1/(f_0(x)^k+\beta(x))^{1+1/k}$. Since this factor is strictly positive it is equivalent to analyze the vector field
\begin{equation}
\dot{x} = \tilde{\nabla}\varphi_k(x) + \tilde{b}_k(x),
\end{equation}
where $\tilde{\nabla}\varphi_k(x) =(f_0(x)^k+\beta(x))^{1+1/k}\nabla \varphi_k(x)$ and $\tilde{b}_k(x) = (f_0(x)^k+\beta(x))^{1+1/k} b^{diff}_k(x)$. Observe that there exists a region, depending on $k$ away of the critical points of $ \varphi_k(x) $ such that it holds that 
\begin{equation}\label{eqn_region_away_of_critical_points}
\left(\tilde{\nabla} \varphi_k(x) + \tilde{b}_k(x) \right)^T \nabla\varphi_k(x) >0.
\end{equation}
To prove the previous statement, observe that $\|\tilde{b}_k\|$ is strictly decreasing with $k$ (c.f. Assumption \ref{assumption_bias}). Therefore, for any $x$ such that $\|\tilde{\nabla}\varphi_k\|$ is bounded away from zero we have that there exists a $K'$ for which $\tilde{\nabla} \varphi_k(x)$ dominates the term $\tilde{b}_k(x)$ and therefore \eqref{eqn_region_away_of_critical_points} holds. In this region the function $\varphi_k(x)$ is strictly decreasing along the flow of the differential equation $\dot{x} =-\left(\tilde{\nabla} \varphi_k(x)+\tilde{b}_k(x)\right)$ and thus there cannot be recurrences. Therefore, it remains to be shown that the flow is gradient like in the neighborhood of the critical points where \eqref{eqn_region_away_of_critical_points} is not satisfied. Observe that there exists two types of critical points, the minimum of $\varphi_k(x)$ and the saddles. Let us focus on the neighborhood around the minimum first. To that end we compute the Jacobian of $\tilde{\nabla}\varphi_k(x)$ 
\begin{equation}
\beta(x)\nabla^2f_0(x) + \nabla \beta(x)\nabla f_0(x)^T\left(1-\frac{1}{k}\right)-\frac{f_0(x)}{k}\nabla^2 \beta(x).
\end{equation}
It can be shown that for every $\varepsilon>0$ there exists $K$ such that if $k>K$ then the minimum of $\varphi_k(x)$ is at a distance smaller than $\varepsilon$ from the minimum of $f_0(x)$ (c.f. Lemma 2 \cite{PaternainEtal15}). Thus $\beta(x)$ is bounded away from zero and thus the first term in the above equation dominates the Jacobian of $\tilde{\nabla}\varphi_k(x)$. In particular, this implies that the eigenvalues of the Jacobian in a neighborhood of the minimum of $\varphi_k(x)$ are of the order $O(k^0)$.
 Thus the region around the minimum where the linearized system is conjugate to the original is also independent of $k$. This means that for large enough $k$ the region where \eqref{eqn_region_away_of_critical_points} near the minimum is contained in the region where the flow $\dot{x}=-\tilde{\nabla}\varphi_k(x)$ is conjugate to the flow of the linearization. Furthermore, in that region the norm of the linearized field is lower bounded by the eigenvalues of $\nabla^2f_0(x)$ and this bound is independent of $k$. On the other hand, we have that $\lim_{k\to\infty}\|\tilde{b}_k(x)\|_{C^1} =0$. Since the minimum of $\nabla \varphi_k(x)$ is non degenerate the flow $\dot{x}=-\tilde{\nabla}\varphi_k(x)$ is structurally stable and therefore for large enough $k$ $\dot{x}=-\tilde{\nabla}\varphi_k(x) -\tilde{b}_k$ is conjugate to $\dot{x}=-\tilde{\nabla}\varphi_k(x)$. Which means that the vector field $\tilde{\nabla}\varphi_k(x)+\tilde{b}_k$ cannot have recurrences in the neighborhood of the minimum of $ \varphi_k(x) $. We are left to show that the same holds true in the neighborhoods of the saddle points of $\varphi_k(x)$. The latter is a direct consequence of Assumption \ref{assumption_bias} and the fact that $\nabla \varphi_k(x)$ is Morse-Smale, therefore structurally stable. The above completes the proof that the vector field $\tilde{\nabla}\varphi_k(x)+\tilde{b}_k(x)$ is gradient-like. Since the original vector field of interest is the one analyzed times a strictly positive function the same holds for it thus completing the proof of the lemma.
%
%
\end{proof}
The above lemma shows that the vector field $\nabla \varphi_k(x)+b^{diff}_k(x)$ is gradient-like, therefore by virtue of Theorem \ref{theo_energy_function_smale} a function $V_k(x)$ satisfying \eqref{lemma_energy_function_eqn1} and \eqref{lemma_energy_function_eqn2} for an estimate $\hat{g}_t(x_t,\theta_t)$ such that $\E{\hat{g}_t(x_t,\theta_t\Big| \ccalG_t} = \alpha(x_t)\left(\nabla \varphi_k(x)+b^{diff}_k(x)\right)$ exists. The same function satisfies \eqref{lemma_energy_function_eqn1} and \eqref{lemma_energy_function_eqn2} for an estimate $\hat{g}_t(x_t,\theta_t)$ with a piece-wise differentiable bias since its discontinuities are away from the obstacles and thus away of the critical points. This completes the proof of the Lemma.

\bibliographystyle{ieeetr}
\bibliography{bib}


%
\begin{IEEEbiography}[]{Santiago Paternain}  
received the B.Sc. degree in electrical engineering from Universidad de la Rep\'ublica Oriental del Uruguay, Montevideo, Uruguay in 2012. Since August 2013, he has been working toward the Ph.D. degree in the Department of Electrical and Systems Engineering, University of Pennsylvania. His research interests include optimization and control of dynamical systems. 
\end{IEEEbiography}
\begin{IEEEbiography}[]{Alejandro Ribeiro}
received the B.Sc. degree in electrical engineering from the Universidad de la Republica Oriental del Uruguay, Montevideo, in 1998 and the M.Sc. and Ph.D. degree in electrical engineering from the Department of Electrical and Computer Engineering, the University of Minnesota, Minneapolis in 2005 and 2007. From 1998 to 2003, he was a member of the technical staff at Bellsouth Montevideo. After his M.Sc. and Ph.D studies, in 2008 he joined the University of Pennsylvania (Penn), Philadelphia, where he is currently the Rosenbluth Associate Professor at the Department of Electrical and Systems Engineering. His research interests are in the applications of statistical signal processing to the study of networks and networked phenomena. His focus is on structured representations of networked data structures, graph signal processing, network optimization, robot teams, and networked control. Dr. Ribeiro received the 2014 O. Hugo Schuck best paper award, the 2012 S. Reid Warren, Jr. Award presented by Penn's undergraduate student body for outstanding teaching, the NSF CAREER Award in 2010, and paper awards at the 2016 SSP Workshop, 2016 SAM Workshop, 2015 Asilomar SSC Conference, ACC 2013, ICASSP 2006, and ICASSP 2005. Dr. Ribeiro is a Fulbright scholar and a Penn Fellow.
\end{IEEEbiography}

\end{document}